\let \pr=\S

\documentclass[12pt,article,reqno,draft]{amsart}
\usepackage{amsfonts,amsmath,amssymb,mathrsfs}

\usepackage[english]{babel}

\newcommand{\al}{\alpha}
\newcommand{\be}{\beta}
\newcommand{\ga}{\gamma}
\newcommand{\de}{\delta}
\newcommand{\la}{\lambda}
\newcommand{\eps}{\varepsilon}
\newcommand{\iy}{\infty}
\newcommand{\Om}{\Omega}
\newcommand{\ol}{\overline}

\renewcommand{\Re}{\mathrm{Re}}
\renewcommand{\Im}{\mathrm{Im}}

\newcommand{\e}{\varepsilon}
\renewcommand{\phi}{\varphi}
\newcommand\wt[1]{{\widetilde{#1}}}

\newcommand\cm{{\mathscr M}}

\newcommand\R{{\mathbb R}}

\newcommand\N{{\mathbb N}}
\newcommand\Z{{\mathbb Z}}
\newcommand\I{{\mathbb I}}
\renewcommand{\S}{\mathbb{S}}

\newcommand\C{{\mathbb C}}

\newcommand\loc{\operatorname{loc}}

\newcommand{\td}{\tilde}
\renewcommand{\td}{\widetilde}
\renewcommand{\tilde}{\widetilde}
\newcommand{\F}{\mathscr F}

\newcommand{\WplOm}{\overset{\circ}{W}\rule{0pt}{2mm}^l_p(\Omega)}
\newcommand{\WplOmInf}{\overset{\circ}{W}\rule{0pt}{2mm}^l_\iy(\Omega)}
\newcommand{\WplRnInf}{\overset{\circ}{W}\rule{0pt}{2mm}^l_\iy(\R^n)}
\newcommand{\WplRn}{\overset{\circ}{W}\rule{0pt}{2mm}^l_p(\R^n)}
\newcommand{\WplRTwo}{\overset{\circ}{W}\rule{0pt}{2mm}_p^l(\R^2)}
\newcommand{\WpTwoRTwoInf}{\overset{\circ}{W}\rule{0pt}{2mm}_\iy^2(\R^2)}
\newcommand{\WplRTwoInf}{\overset{\circ}{W}\rule{0pt}{2mm}_\iy^l(\R^2)}

\newtheorem{theorem}{Theorem}[section]
\newtheorem{lemma}[theorem]{Lemma}
\newtheorem{cor}[theorem]{Corollary}
\newtheorem{prop}[theorem]{Proposition}

\numberwithin{equation}{section}

\theoremstyle{definition}
\newtheorem{definition}[theorem]{Definition}
\newtheorem{remark}[theorem]{Remark}

\renewcommand{\(}{\left(}
\renewcommand{\)}{\right)}
\renewcommand{\[}{\left[}
\renewcommand{\]}{\right]}

\newcommand{\PjxD}{{\{P_j(x,D)\}_1^N}}
\newcommand{\Pj}{{\{P_j(D)\}_1^N}}

\DeclareMathOperator{\rank}{rank}
\DeclareMathOperator{\Span}{span} \DeclareMathOperator{\ind}{ind}
\DeclareMathOperator{\const}{const}
\DeclareMathOperator{\sign}{sign} \DeclareMathOperator{\dom}{dom}

\renewcommand{\le}{\leqslant}
\renewcommand{\ge}{\geqslant}

\begin{document}



\author{D.~V.~Limanski\v{i} and M.~M.~Malamud}

\title
[Elliptic and weakly coercive systems]
  {Elliptic and weakly coercive systems \\
  of operators in Sobolev spaces}


\begin{abstract}
It is known that an elliptic system $\{P_j(x,D)\}_1^N$ of order
$l$ is weakly coercive in
$\overset{\circ}{W}\rule{0pt}{2mm}^l_\infty(\mathbb R^n)$, that
is, all differential monomials of order $\le l-1$ on
$C_0^\infty(\mathbb R^n)$-functions are subordinated to this
system in the $L^\infty$-norm. Conditions for the converse result
are found and other properties of weakly coercive systems are
investigated.

An analogue of the de Leeuw-Mirkil theorem is obtained for
operators with variable coefficients: it is shown that an operator
$P(x,D)$ in $n\ge 3$ variables with constant principal part is
weakly coercive in
$\overset{\circ}{W}\rule{0pt}{2mm}_\infty^l(\mathbb R^n)$ if and
only if it is elliptic. A similar result is obtained for systems
$\{P_j(x,D)\}_1^N$ with constant coefficients under the condition
$n\ge 2N+1$ and with several restrictions on the symbols
$P_j(\xi)$ .

A complete description of differential polynomials in two
variables which are weakly coercive in
$\overset{\circ}{W}\rule{0pt}{2mm}_\infty^l(\mathbb R^2)$ is
given. Wide classes of systems with constant coefficients which
are weakly coercive in
$\overset{\circ}{W}\rule{0pt}{2mm}_\infty^l(\mathbb\R^n)$, but
non-elliptic are constructed.

Bibliography: 32 titles.
\end{abstract}

\maketitle
\section{Introduction}

Let $\Omega$ be an arbitrary domain in $\R^n$, let $p\in[1,\iy]$,
and let $l:=(l_1,\dots,l_n)$ be a vector with positive integer
components. In $L^p(\Om)$ consider  a system $\PjxD$ of
differential operators of the form
\begin{equation} \label{8}
P_j(x,D)= \sum_{|\al:l|\le 1} a_{j\al}(x) D^\al,\qquad
j\in\{1,\dots,N\},
    \end{equation}
with measurable coefficients $a_{j\al}(\cdot)$. Further, let
$P_j^l(x,D):= \sum_{|\al:l|=1} a_{j\al}(x) D^\al$ be the
$l$-principal part of the operator $P_j(x,D)$, and let
 $P_j^l(x, \xi):= \sum_{|\al:l|=1} a_{j\al}(x) \xi^\al$ be its
 principal $l$-quasihomogeneous symbol. We recall the following
 definition.

\begin{definition} (see \cite{BIN}-\cite{VolGin}) \label{def_quasiell} A
system of differential operators $\PjxD$ of the form~\eqref{8} is
said to be \textit{$l$-quasielliptic} if
\begin{equation*}
   \(P_1^l(x,\xi),\dots,P_N^l(x,\xi)\)\neq 0,\qquad
    (x,\xi)\in\Om\times\(\R^n\setminus\{0\}\).
\end{equation*}
In particular, if $l_1=\dots=l_n=l$, then it is called an
\textit{elliptic system of order $l$}.
\end{definition}

As is known, an elliptic operator of order $l$ does not exist for
every $l$. Using a result due to Lopatinski\v{i}~\cite{Lop} (see
also~\cite{Agr}, \cite{lions}, Ch.\,2, \pr\,1, \cite{tribel}), for
$n\ge 3$ an elliptic operator $P(D)$ is properly elliptic and, in
particular, has even order. To the best of our knowledge, a
similar problem for $l$-quasielliptic operators remains unsolved
at present. In \pr\,3, using the Borsuk-Ulam theorem
(Theorem~\ref{borsuk}), we obtain  a complete description of those
$l$ for which $l$-quasielliptic systems exist. Namely, the
following theorem holds.

\begin{theorem}\label{exist_quasi_system}
Let $l=(l_1,\dots,l_n)\in\N^n$ and let $n\ge 2N+1$. Then
$l$-quasielliptic systems $\PjxD$ of the form~\eqref{8} exist if
and only if the number of odd integers among $l_1,\dots,l_n$ does
not exceed $2N-1$.
\end{theorem}

Let $\PjxD$ be a system of differential operators of the
form~\eqref{8} with coefficients $a_{j\al}(\cdot)\in
L_{\loc}^\iy(\Om)$. We recall the following notion.

\begin{definition} \label{coerc_def}
(see \cite{BIN}, Ch.\,3, \pr\,11.1) A system of differential
operators $\PjxD$ of the form~\eqref{8} is said to be
\textit{coercive} in the (anisotropic) Sobolev space $\WplOm$,\
$p\in[1,\iy]$, if the following estimate holds:
\begin{equation}  \label{10}
\|f\|_{W_p^l(\Om)}:=\sum_{|\al:l|\le 1} \|D^\al f\|_{L^p(\Om)}\le
C_1 \sum_{j=1}^N \|P_j(x,D)f\|_{L^p(\Om)} +C_2 \|f\|_{L^p(\Om)},
\end{equation}
where $C_1$ and $C_2$ do not depend on $f\in C_0^\iy(\Om)$.
\end{definition}

It is well known (see~\cite{BIN}, \cite{VolGin}, \cite{Bes} and
\cite{Horm}) that, under some constraints on the coefficients
$a_{j\al}(\cdot)$ and on the domain $\Om$ the system~\eqref{8} is
$l$-quasiellitpic if and only if it is coercive in $\WplOm$ for
$p\in(1,\iy)$. If $p=1$ or $\iy$ then the estimate~\eqref{10} does
not hold any longer for an $l$-quasielliptic system. Namely, the
following assertion was proved by one of the authors of this
paper.

\begin{prop} \label{mal_neodn}
\emph{(see \cite{MalUMZh}-\cite{Mal2}, \pr\,5, Theorem 3)} Let
$\Om$ be a domain in $\R^n$, and let $Q(x,D)$ and $\PjxD$ be
differential operators of the form
\begin{equation} \label{Q_Pj}
  Q(x,D)=\sum_{|\al:l|\le 1} b_\al(x) D^\al,\qquad
  P_j(x,D)=\sum_{|\al:l|\le 1} a_{j\al}(x) D^\al,
\end{equation}
where $x\in\Om$, $j\in\{1,\dots,N\}$, and the coefficients
$a_{j\al}(\cdot), b_\al(\cdot)\in L^\iy_{\loc}(\Om)$ for
$|\al:l|<1$ and $a_{j\al}(\cdot), b_\al(\cdot)\in C^1(\Om)$ for
$|\al:l|=1$. Then the estimate
\begin{equation}  \label{200}
\|Q(x,D)f\|_{L^p(\Om)}\le C_1 \sum_{j=1}^N
\|P_j(x,D)f\|_{L^p(\Om)}+C_2 \|f\|_{L^p(\Om)},\qquad f\in
C_0^\iy(\Om),
 \end{equation}
for $p=\iy$ yields the equality
\begin{equation}\label{zz}
Q^l(x,\xi)=\sum_{j=1}^N \la_j(x) P_j^l(x,\xi),\qquad
x\in\Om,\qquad \xi\in\R^n,
\end{equation}
in which $\la_j(\cdot)\in C^1(\Om)$. If the operators $Q(x,D)$ and
$\PjxD$ have constant coefficients, then the functions $\la_j(x)$
in~\eqref{zz} are also constant: $\la_j(x)\equiv \la_j$.
\end{prop}

A criterion for the system $\PjxD$ to be coercive in $\WplOmInf$
was found in~\cite{Mal1}, \cite{Mal2} (in the isotropic case it
was found earlier in~\cite{MalUMZh}). This criterion yields that
an $l$-quasielliptic system is coercive in $\WplOmInf$ only in
exceptional cases. Nevertheless, for an $l$-quasielliptic system
$\{P_j\}_1^N$ the following estimate holds:
\begin{equation}  \label{10'}
\sum_{|\al:l|<1} \|D^\al f\|_{L^p(\Om)}\le C_1 \sum_{j=1}^N
\|P_j(x,D)f\|_{L^p(\Om)} +C_2 \|f\|_{L^p(\Om)},\qquad f\in
C_0^\iy(\Om).
\end{equation}
For  $p\in (1, \iy)$ this estimate is implied by the
estimate~\eqref{10} established in~\cite{BIN} and~\cite{Bes} (see
also~\cite{VolGin}) and for $p=\iy$ it is proved in~\cite{Mal1}
and~\cite{Mal2}. Note also that the fact that the
estimate~\eqref{10} is impossible in the case $p=1$ follows from a
result due to Ornstein~\cite{Orn}. But in the case $p=1$, the
estimate~\eqref{10'} was proved for operators with constant
coefficients in~\cite{BDLM} and~\cite{dan2004}.

These results suggest the following natural definition introduced
in~\cite{dan2004}.

\begin{definition} \label{weak_coerc}
A system of differential operators  $\PjxD$ of the form~\eqref{8}
is said to be \textit{weakly coercive} in the anisotropic Sobolev
space $\WplOm$,\; $p\in[1,\iy]$, if the estimate~\eqref{10'} is
valid with $C_1$ and $C_2$ independent of $f$.

In the case of isotropic Sobolev space  $\WplOm$, that is, for
$l_1=\dots=l_n=l$, the inequality  $|\al:l|<1$\ in~\eqref{10'}
takes the usual form $|\al|<l$.
\end{definition}

In the case of one operator de Leeuw and Mirkil~\cite{LeeMir}
showed before that for $n\ge 3$ an elliptic operator $P(D)=P_1(D)$
can be characterized by means of a priori estimates in
$L^\iy(\R^n)$.

\begin{theorem} \emph{(see \cite{LeeMir}, p.\,119)} \label{deLeu-Mir_intr}
Assume that $n\ge 3$. Then the  ellipticity of a differential
operator $P(D)$ of order  $l\ge 2$ is equivalent to its weak
coercivity in $\WplRnInf$.
\end{theorem}

The condition $n\ge 3$ is essential in
Theorem~\ref{deLeu-Mir_intr}. In fact, Malgrange presented an
example of a non-elliptic operator $P(D)=(D_1+i)(D_2+i)$ that is
weakly coercive in $\overset{\circ}{W}\rule{0pt}{2mm}_\iy^2(\R^2)$
(see~\cite{LeeMir}, p.\,123).

In this paper we mainly  consider  homogeneous systems $\PjxD$ of
the form~\eqref{8} consisting of operators with homogeneous
principal symbols of order $l$. Our investigation of the
quasihomogeneous case  is postponed till the next publication. To
avoid the possibility of repetition here we present only those
'anisotropic' results whose proofs do not differ in practice from
the corresponding 'isotropic' ones.

A considerable proportion of our results is relate the de
Leeuw-Mirkil Theorem~\ref{deLeu-Mir_intr}. Namely, we extend
Theorem~\ref{deLeu-Mir_intr} to a system $\Pj$ with constant
coefficients (Theorem~\ref{odnorod}) and also prove its analogue
for an operator $P(x,D)$ with variable coefficients
(Theorem~\ref{de_Leeuw_vary}). To prove the latter we use a new
method which is essentially based on Proposition~\ref{mal_neodn}
and also on some topological concepts (summarized in
Proposition~\ref{prop_weak_coerc}, (iii)). Note that the method
in~\cite{LeeMir} is not applicable to operators with variable
coefficients, although in proving Theorem~\ref{odnorod}, which
concerns systems with constant coefficients, alongside the
topological concepts we use some arguments from~\cite{LeeMir}.

In addition, we present a complete description of weakly coercive
operators of two variables in
$\overset{\circ}{W}\rule{0pt}{2mm}_\iy^l(\R^2)$
(Theorems~\ref{obschij_vid} and~\ref{th_16}). In particular, in
doing this we show that the non-trivial zeros of the principal
symbol of a weakly coercive operator are simple
(Proposition~\ref{prop_weak_coerc}, (iv)). Note that to prove this
last result, as well as in the proof of
Theorem~\ref{de_Leeuw_vary}, we use an analogue of
Theorem~\ref{deLeu-Mir_intr}, an anisotropic version of
Proposition~\ref{mal_neodn}. This application of
Proposition~\ref{mal_neodn} to the proof of 'isotropic' results is
based on the possibility, in principle, of a non-unique selection
of the principal part of a differential operator.

Note also that topological arguments are also used in \pr\,4, to
prove an analogue of Theorem~\ref{exist_quasi_system} in the case
of a weakly coercive system (Theorem~\ref{prop6}). Namely,
invoking Borsuk's theorem (Theorem~\ref{borsuk_2}) and degree
theory we show that, under some restrictions, the system $\PjxD$
has even order.

It is also worth mentioning that in \pr\,6, in the construction of
weakly coercive, but non-elliptic systems, new non-symmetric
multipliers on $L^p$, $p\in[1,\iy]$, arise, which are not
traditional in elliptic theory. For instance, it is shown in the
proof of Theorem~\ref{th_4.3} that if $P(\xi)$ is an elliptic
polynomial of degree $l$, then
$$
m(\xi):=\chi(\xi)\frac{\xi^\al}{P(\xi)\sum_{k=2}^n
(1+\xi_k^2)}\in\cm_1(\R^n) \qquad \text{for}\quad |\al|\le
l+1,\quad \al_1\le l-1,
$$
that is, $m(\cdot)$ is a multiplier on $L^1(\R^n)$, hence a
multiplier on $L^p(\R^n)$ for $p\in[1,\iy]$. Here $\chi(\cdot)$ is
a suitable 'cutoff' function. To verify the inclusion
$m\in\cm_p(\R^n)$ for $p\in(1,\iy)$ we can use the
Mikhlin-Lizorkin theorem (see~\cite{BDLM}, and also~\cite{Liz}
and~\cite{Mih}), but this is insufficient for verifying the
inclusion $m\in\cm_1$. To prove the latter we use a result on
multipliers from~\cite{BDLM}.

The paper is organized as follows. In \pr\,2 we present auxiliary
topological and analytic results necessary in what follows. In
\pr\,3 we prove the existence criterion for $l$-quasielliptic
systems (Theorem~\ref{exist_quasi_system}) and a stability
criterion  for systems of order $l$ under perturbations of order
$\le l-1$ (Proposition~\ref{har_ellipt}). We devote \pr\,4 to
properties of weakly coercive systems in the isotropic spaces
$\WplRn$. We also prove there analogues of
Theorem~\ref{deLeu-Mir_intr} for the case of a homogeneous system
(Theorem~\ref{odnorod}) and that of an operator with variable
coefficients (Theorem~\ref{de_Leeuw_vary}). In \pr\,5 we give a
complete description of operators in two variables that are weakly
coercive in $\overset{\circ}{W}\rule{0pt}{2mm}_\iy^l(\R^2)$, but
are not elliptic (Theorems~\ref{obschij_vid} and~\ref{th_16}).
Finally, \pr\,6 is devoted to describing wide classes of
non-elliptic systems that are weakly coercive in the isotropic
space $\WplRnInf$ (Theorem~\ref{th_4.3}).

A part of the results here were announced (without proofs)
in~\cite{dan2004} and~\cite{dan2007}.

We would like to express our sincere gratitude to L.~R.~Volevich
with whom we repeatedly discussed the results of the work. We are
also grateful to O.~V.~Besov, L.~D.~Kudryavtsev, S.~I.~Pokhozhaev,
as well as to all participants of their seminar, at which this
work was presented, and also to L.~L.~Oridoroga, for useful
discussions. Finally, we are deeply thankful to the referee, who
read this manuscript very carefully and pointed out several
mistakes in its original version.

We devote this work to the blessed memory of L.~R.~Volevich, a
remarkable person and mathematician. M.~M.~Malamud was a close
friend of L.~R.~Volevich, who had a significant influence on his
understanding of elliptic theory.

\section{Preliminaries}

We will use the following notation. Let $\Z_+:=\N\cup\{0\}$, let
$\Z^n_+:=\Z_+\times\dots\times\Z_+$ ($n$ is the number of
factors), and $\Z_2:=\{0,1\}$. Further, let
$D_k:=-i{\partial}/{\partial x_k}$ and $D=(D_1, D_2, \dots, D_n)$;
for a multi-index $\al=(\al_1, \dots, \al_n)\in \Z^n_+$ we set
$|\al|:=\al_1+\dots+\al_n$ and $D^\al:=D^{\al_1}_1 D^{\al_2}_2
\dots D^{\al_n}_n$. If $l=(l_1, \dots,l_n)\in\N^n$ and $\alpha \in
{\Z}_+^n$, then $|\al:l|:=\al_1 / l_1+ \dots + \al_n/l_n$.

Also let $|x|:=(\sum_1^n x_k^2)^{1/2}$, $\langle
x,y\rangle:=\sum_1^n x_k y_k$, where $x=(x_1,\dots,x_n)$,
$y=(y_1,\dots,y_n)$, $x,y\in\R^n$. Denote by
$\S_r^n:=\{x\in\R^{n+1}: |x|=r\}$ the $n$-dimensional sphere of
radius $r$ in $\R^{n+1}$, with $\S^n:=\S_1^n$; and by
$B_r^n:=\{x\in\R^n: |x|\le r\}$ the closed ball of radius $r$.

We denote by $\I=\I_n$ the identity operator in $\R^n$ and by
$\cm_p=\cm_p(\R^n)$ the algebra of multipliers on $L^p(\R^n)$,\
$p\in[1,\iy]$.

\subsection{Topological concepts}

\begin{theorem} \label{borsuk}
\emph{(the Borsuk-Ulam theorem; see~\cite{spenyer}, Ch.\,5,
\pr\,8.9)} For each continuous mapping $f:\S^n\to\R^n,\ n\ge 1$,
there is a point $x\in\S^n$ such that $f(x)=f(-x)$.
\end{theorem}

Following~\cite{spenyer}, Ch.\,4, \pr\,7, and~\cite{nirenb} recall
the notion of the degree of a map. As is known, the
$n$-dimensional homotopy group of the sphere $\S^n$ is isomorphic
to $\Z$,\; $\pi_n(\S^n) \simeq \Z$. Each continuous map
$f:\S^n\to\S^n$ induces a group homomorphism
$f_*:\pi_n(\S^n)\to\pi_n(\S^n)$, hence $f_*:\Z\to k\Z$. The
integer $k$ does not depend on the choice of a generator of the
group $\pi_n(\S^n)$; it is referred to as \emph{the degree of $f$}
and is denoted by $\deg f$.

Since the $n$-dimensional homology group $H_n(\S^n;{\Z})
\simeq{\Z}$, the degree of a map $f:\ \S^n\to \S^n$ can be defined
in terms of the homomorphism $f_{* n}:\ H_n(\S^n;{\Z})\to
H_n(\S^n;{\Z})$. These definitions are equivalent.

Further, homotopic maps have equal degree. The converse also holds
(Hopf's theorem).

Since $\R^{n+1}\setminus\{0\}$ is homotopy equivalent to $\S^n$,
it follows that $\pi_n(\R^{n+1}\setminus\{0\}) \simeq \pi_n(\S^n)$
and so maps $f:\S^n\to\R^{n+1}\setminus\{0\}$ have well defined
degrees.

We will use the following statements repeatedly.

\begin{theorem} \emph{(see~\cite{nirenb}, \pr\,1.4)}
\label{step_otobr} A continuous map
$f:\S^n\to\R^{n+1}\setminus\{0\}$ can be extended to a continuous
map of the closed ball $B_1^{n+1}$ into $\R^{n+1}\setminus\{0\}$
if and only if $\deg f=0$.
\end{theorem}

\begin{theorem} \label{borsuk_2}
\emph{(Borsuk's theorem on the degree of a map; see~\cite{nirenb},
\pr\,1.7)} Let $f$ be an odd map of the sphere $\S^n$ into inself:
$f(-x)=-f(x)$. Then its degree $\deg f$ is odd.
\end{theorem}

\subsection{Analytic results}

\begin{lemma} \emph{(Eberlein's theorem; see~\cite{LeeMir}, p.\,114)}
\label{eberlein} Let $\mu$ be a finite Borel measure on $\R^n$ and
$M$ the Fourier-Stieltjes transform of $\mu$. Then the constant
function $c\equiv\mu(0)$ can be uniformly approximated by
functions of the form $\pi\ast M$, where $\pi$ is a probability
measure (that is, $\pi(\R^n)=1$).
\end{lemma}

\begin{prop}\label{Q=MP+N}
\emph{(see~\cite{LeeMir}, p.\,113, Proposition 1)} Let $Q(D)$ and
$\Pj$ be differential operators of the form~\eqref{Q_Pj} with
constant coefficients. Then the estimate~\eqref{200} for $p=\iy$
and $\Om=\R^n$ is equivalent to the identity
\begin{equation} \label{zzz}
 Q(\xi)=\sum_{j=1}^N M_j(\xi) P_j(\xi) + M_{N+1}(\xi),\qquad
 \xi\in\R^n,
\end{equation}
where the $\{M_j(\cdot)\}_1^{N+1}$ are the Fourier-Stieltjes
transforms of finite Borel measures on $\R^n$.
\end{prop}

\begin{prop}\label{c_j} \emph{(see~\cite{LeeMir}, p,\,114)}
Let $\Pj$ be a system satisfying estimate~\eqref{200} with $p=\iy$
and $\Om=\R^n$, and assume that the principal forms
$\{P_j^l(\xi)\}_1^N$ are linearly independent. Also let
$\{\la_j\}_1^N$ be the coefficients in equality~\eqref{zz} and
$\{\mu_j\}_1^N$ be finite Borel measures, $\hat\mu_j=M_j$, where
$\{M_j(\cdot)\}_1^N$ are the functions in~\eqref{zzz}. Then
$\la_j=\mu_j(0)$,\ $j\in\{1,\dots,N\}$.
\end{prop}

The next statement is well known to experts. Moreover, it was
mentioned (without proof) in~\cite{Bom}. For the sake of
completeness we present it here with the proof.

\begin{prop}\label{prop_alg_ineq}
Let $p\in[1,\iy]$ and let $Q(D)$ and $\Pj$ be differential
operators of the form~\eqref{Q_Pj} with constant coefficients.
Then the a priori estimate~\eqref{200} implies the following
algebraic inequality for the symbols:
\begin{equation}\label{alg_ineq}
|Q(\xi)|\le C'_1\sum_{j=1}^N |P_j(\xi)|+C'_2,\qquad \xi\in\R^n.
\end{equation}
\end{prop}

\begin{proof}[Sketch of the proof]
(i) Let $p\in[1,\iy)$ and let $\psi\in C_0^\iy(\R^n)$,\
$\psi\not\equiv 0$. We set $\psi_r(x):=\psi(x/r) =
\psi(x_1/r,\dots,x_n/r)$,\ $r>0$. It can be verified directly that
\begin{equation} \label{2'''}
\|D^\al\psi_r\cdot P(D) e^{i\langle x,\xi\rangle}\|_p=
r^{-|\al|}\cdot r^{n/p}|P(\xi)|\cdot\|D^\al\psi\|_p.
\end{equation}
Applying Leibniz's formula  to $f_r(x):=\psi_r(x) e^{i\langle
x,\xi\rangle}$ (see~\cite{Horm}, Ch.\,II, \pr\,2.1) we obtain
\begin{equation} \label{herm}
P(D) f_r=P(D)(\psi_r e^{i\langle x,\xi\rangle})=\sum_\al
\(\frac{1}{|\al|!}\) D^\al \psi_r \cdot P^{(\al)}(D) e^{i\langle
x,\xi\rangle},
\end{equation}
where $P^{(\al)}(D)$ is the operator with symbol $D^\al P(\xi)$.
Taking~\eqref{2'''} and~\eqref{herm} into account we obtain
\begin{gather*}
\|P_j(D)f_r\|_p\le r^{n/p}|P_j(\xi)|\cdot\|\psi\|_p+o\(r^{n/p}\),
\\ \|Q(D)f_r\|_p\ge
r^{n/p}|Q(\xi)|\cdot\|\psi\|_p+o\(r^{n/p}\).
\end{gather*}

Substituting the above expressions in~\eqref{200} we arrive at the
estimate
\begin{equation}\label{3'''}
r^{n/p}|Q(\xi)|\cdot\|\psi\|_p+o\(r^{n/p}\) \le C'_1\sum_{j=1}^N
r^{n/p}|P_j(\xi)|\cdot\|\psi\|_p+o\(r^{n/p}\)+C'_2
r^{n/p}\|\psi\|_p.
 \end{equation}
Dividing both sides of~\eqref{3'''} by $r^{n/p}\|\psi\|_p>0$ and
then passing to the limit as $r\to\iy$ we obtain~\eqref{alg_ineq}.

(ii) For $p=\iy$ the proof is similar: it suffices to note that
for $p=\iy$ the factor $r^{n/p}$ in~\eqref{2'''} is equal to
$r^{n/\iy}=r^0=1$.
\end{proof}

\begin{remark} \label{rem_mal_neodn}
(i) Proposition~\ref{mal_neodn} (together with its proof) remains
valid if all the differential monomials $D^\al$ depend only on the
components $D_1,\dots,D_m$, in accordance with the decomposition
$\R^n=\R^m\oplus\R^{n-m}$, that is, $D^\al=D_1^{\al_1}\dots
D_m^{\al_m}\otimes\I_{n-m}$. In this case equality~\eqref{zz}
follows from inequality~\eqref{200} with $p=\iy$ in which $f\in
C_0^\iy(\Om_m)$, where $\Om_m=\pi_m\Om$ is the projection of the
domain $\Om$ onto $\R^m$.

(ii) Inequality~\eqref{alg_ineq} is a consequence of the
estimate~\eqref{200}, but is not equivalent to~\eqref{200}. For
instance, if $P(\xi)=\xi_1^2+\xi_2^2$ and
$Q(\xi)=\xi_1^2-\xi_1\xi_2+\xi_2^2$, then $Q(\xi)< \frac 3 2
P(\xi)$. At the same time, in view of Proposition~\ref{mal_neodn},
the estimate~\eqref{200} does not hold for $p=\iy$.

(iii) Proposition~\ref{prop_alg_ineq} fails for arbitrary domains
$\Om$. For instance, if $\Om$ is bounded and $p=2$, then, by a
theorem of H\"ormander's in~\cite{Horm}, \pr\,2.3
estimate~\eqref{200} holds for $P(D)=D_1^2-D_2^2$ and $Q(D)=D_1$,
whereas the inequality $|\xi_1|\le C\[|\xi_1^2-\xi_2^2|+1\]$
fails.
\end{remark}

\subsection{Multipliers}

\begin{definition} (see~\cite{St}, Ch.\,IV, \pr\,3) \label{mult_def}
Let $\F$ be the Fourier transform in $L^2(\R^n)$. A bounded
(Lebesgue-)measurable function $\Phi : \R^n\to\C$ is called a
\emph{multiplier} on $L^p(\R^n)$,\ $p\in[1,\iy]$, if the
convolution operator $f\to T_{\Phi} f =:\F^{-1}\Phi \F f$ takes
$L^p(\R^n)\cap L^2(\R^n)$ to $L^p(\R^n)$ and is bounded in
$L^p(\R^n)$.
 \end{definition}

A simple description of the spaces $\cm_p$ of multipliers on
$L^p(\R^n)$ is known only for $p=1,2,\iy$. In particular,
$\cm_1=\cm_\iy$ is the set of images under the Fourier-Stieltjes
transform of finite Borel measures on $\R^n$ (see~\cite{St},
Ch.\,IV, \pr\,3)
\begin{equation}\label{2.7A}
\Phi\in\cm_1 =\cm_\iy  \quad\Longleftrightarrow\quad \Phi(\xi)=
\hat{\mu}(\xi) :=\int_{\R^n}
 e^{i\langle x,\xi\rangle}\,d\mu(x),\qquad \mu(\R^n)<\iy.
\end{equation}
For other values $p\in (1,\iy)$ only sufficient conditions are
known for the inclusion $\Phi\in {\cm}_p$ to hold (see~\cite{BIN},
\cite{tribel}, \cite{St}). Note that $\cm_1\subset\cm_p$ for
$p\in(1,\iy)$ by~\eqref{2.7A}.

We shall need the following result of~\cite{BDLM} on multipliers
on $L^1$ which in appearance is a (fairly rough) analogue of the
Mikhlin-Lizorkin theorem (see~\cite{BIN}).

\begin{theorem} \emph{(see~\cite{BDLM}, \pr\,3, Theorem 2)} \label{th1}
Let $\Phi\in C(\R^n)$ and assume that for some constants
$\delta\in(0,1)$ and $A_{\delta}>0$,\ $\Phi$ satisfies the
following conditions:
\begin{equation} \label{dv1}
(i) \quad \quad \prod^n_{j=1}(1+|\xi_j|)^{\delta}| \Phi(\xi)|\le
A_{\delta}, \qquad \xi\in \R^n;
\end{equation}

(ii) for all multi-indices $\al,\be\in\Z_2^n$ such that $\al+\be
=(1, 1, \dots, 1)$, there exist derivatives $D^{\al}\Phi$ and
\begin{equation} \label{dv2}
\prod_{j\in \N_\al} |\xi_j|^{1-\delta} (1+|\xi_j|^{2\delta})
\prod_{j\in \N_\be} (1+|\xi_j|)^{\delta} |D_1^{\al_1}\dots
D_n^{\al_n} \Phi(\xi)| \le A_{\delta},\qquad \xi\in\R^n.
\end{equation}
Here $\N_\al\subset\{1,\dots,n\}$ is the support of the
multi-index $\al=(\al_1,\dots,\al_n)\in\Z_+^n$, that is, the set
of subscripts $j\in\{1,\dots,n\}$ for which $\al_j>0$.

Then $\Phi\in \cm_1$, and hence $\Phi\in \cm_p$ for $p\in[1,\iy]$.
\end{theorem}

Note that more general results on multipliers on $L^1(\R^n)$ can
be found in~\cite{Bes86}-\cite{Trigub}, and in~\cite{BelTri},
Theorem 6.4.2. However, in applications to estimates of
differential operators the functions $\Phi$ are usually rational
functions of $\xi$ and $\bar{\xi}$. In this case
conditions~\eqref{dv1} and~\eqref{dv2} can be verified as readily
as the corresponding conditions in the Mikhlin-Lizorkin theorem.

\subsection{Properties of $l$-quasielliptic systems}

The following properties of $l$-quasielliptic systems are
well-known (see~\cite{BIN}, \cite{Vol}, \cite{VolGin85}).

\begin{prop}\label{dvei_neodn}
Let $l=(l_1,\dots,l_n)\in\N^n$ and let $\Pj$ be an
$l$-quasielliptic system of the form~\eqref{8} with constant
coefficients. Then

(i) the zero set of the system $\{P_j(\xi)\}_1^N$ is compact and
hence is contained in some ball $B_r^n$;

(ii) the following two-sided estimate holds:
\begin{equation}\label{est_2}
  C_1\sum_{k=1}^n |\xi_k|^{2l_k}\le
  \sum_{j=1}^N |P_j^l(\xi)|^2 \le
  C_2\sum_{k=1}^n |\xi_k|^{2l_k},\qquad
  C_1,\ C_2>0,\quad \xi\in\R^n.
\end{equation}
\end{prop}

\section{Elliptic and quasielliptic systems}

\subsection{For which values of $l$ do $l$-quasielliptic systems exist?}

Here we prove Theorem~\ref{exist_quasi_system} which was stated in
the introduction, describing all the sets
$l=(l_1,\dots,l_n)\in\N^n$ for which there exist $l$-quasielliptic
systems.

\begin{proof}[Proof of Theorem~\ref{exist_quasi_system}]
\textsl{Necessity.} Let $n=2N+1$.

 (i) Assume first that all the $l_j$ are odd. We claim that
$P_j^l(x,-\xi)=-P_j^l(x,\xi)$. Let $\al\in\N^n$ and let
$\al_1/l_1+\dots+\al_n/l_n=1$. Since all the $l_j$ are odd, this
equality acquires the form $\al_1k_1+\dots+\al_nk_n=k_0$, where
all $k_j$ are odd, $k_j=2k_j'+1$, $j\in\{0,1,\dots,n\}$. Then
$\al_1+\dots+\al_n=2k_0'+1-\sum_{j=1}^n 2\al_jk_j'=2p+1$.
Therefore,
$$
P_j^l(x,-\xi)=\sum_{|\al:l|=1}
a_{j\al}(x)(-\xi)^\al=\sum_{|\al:l|=1} a_{j\al}(x)(-1)^{|\al|}
\xi^\al=-P_j^l(x,\xi).
$$

Now choosing fixed $x_0\in\Om$ we consider the map
$T:=\(T_1,\dots,T_{2N}\): \S^{2N}\to\R^{2N}$, where
\begin{equation}\label{T}
T_{2j-1}(\xi):=\Re\ P_j^l(x_0,\xi),\qquad
 T_{2j}(\xi):=\Im\ P_j^l(x_0,\xi),\qquad j\in\{1,\dots,N\}.
\end{equation}
This map is odd: $T(-\xi)=-T(\xi)$, and by Theorem~\ref{borsuk} we
have $T(\xi^0)=0$ at some point $\xi^0\in\S^{2N}$. But this
contradicts the assumption that the system $\PjxD$ is
$l$-quasielliptic.

(ii) Suppose that precisely one $l_j$ is even; for example, let
$l_1=2^m l_1'$, and assume that $l_1'$ and the other $l_j$ are
odd, $j\in\{2,\dots,n\}$. We claim that the relation $|\al:l|=1$
implies that $\al_1$ is also divisible by $2^m$, that is,
$\al_1=2^m\al_1'$,\ $\al_1'\in\N$. In fact, $|\al:l|=1$ reduces
to the equality
$$
   \al_1 k_1+2^m(\al_2k_2+\dots+\al_nk_n)=2^mk_0,
$$
where all the $k_j$ are odd. Hence $\al_1=2^m\al_1'$.

Further, let $l':=(l_1',l_2,\dots,l_n):=(l_1',l_2',\dots,l_n')$
and let $|\al:l|=1$. Then $\al_1=2^m\be_1$. Setting
$\eta_1=\xi_1^{2^m}$, $\eta_j=\xi_j$ and $\be_j=\al_j$,
$j\in\{2,\dots,n\}$, we write the polynomial $P_j^l(\xi)$ in the
form
$$
  P_j^l(\xi)=\td{P}_j^{l'}(\eta),\qquad\text{where}\quad
  \td{P}_j^{l'}(\eta):=\sum_{|\be:l'|=1} a_\be\eta^\be,\quad
  j\in\{1,\dots,N\}.
$$
The system $\{\td{P}_j^{l'}(\eta)\}_1^N$ is $l'$-quasielliptic. In
fact, let $\td{P}_j^{l'}(\eta^0)=0$,\ $j\in\{1,\dots,N\}$,\
$\eta^0=(\eta_1^0,\dots,\eta_n^0)\in\R^n\setminus\{0\}$. Since all
the $l_j'$ are odd, it follows by (i) that the
$\td{P}_j^{l'}(\eta)$ are odd,
$\td{P}_j^{l'}(-\eta^0)=-\td{P}_j^{l'}(\eta^0)=0$. Therefore, we
may assume that $\eta_1^0>0$. In this case, setting
$\xi_1^0:=(\eta_1^0)^{1/2^m}$ and $\xi_j^0:=\eta_j^0$ for
$j\in\{2,\dots,n\}$ we obtain $P_j^l(\xi^0)=0$,\
$j\in\{1,\dots,N\}$, so that $\xi^0=0$. The last relation
contradicts the assumption that $\eta^0\neq 0$. Thus, the system
$\{\td{P}_j^{l'}\}_1^N$ is $l'$-quasielliptic, where all the
$l_j'$ are odd. By (i) this is impossible. It follows that there
must be at least two even integers among the $l_j$, that is, we
have proved the theorem for $n=2N+1$.

(iii) Suppose that $n>2N+1$, but there are more than $2N-1$ odd
integers among the $l_j$. We assume without loss of generality
that $l_1,\dots,l_{2N}$ are odd. It is clear that the 'restricted'
system $\{P_j(x_0,\xi_1,\dots,\xi_{2N+1},0,\dots,0)\}_1^N$ is
$l'$-quasielliptic, where $l'=(l_1,\dots,l_{2N+1})$. Therefore, by
(i) and (ii) we arrive at a contradiction. Thus, there are at most
$2N-1$ odd integers among $l_1,\dots,l_n$.

\textsl{Sufficiency.} Let $n=2N+1$,\ $l=(l_1,\dots,l_n)$, where
$l_1,\dots,l_{n-2}$ are odd and $l_{n-1}, l_n$ are even. Then the
system
$$
P_1(\xi)=\xi_1^{l_1}+i\xi_2^{l_2},\quad \dots,\quad P_{N-1}(\xi)=
\xi_{n-4}^{l_{n-4}}+ i\xi_{n-3}^{l_{n-3}},\quad
P_N(\xi)=i\xi_{n-2}^{l_{n-2}}+\xi_{n-1}^{l_{n-1}}+\xi_n^{l_n}
$$
is $l$-quasielliptic and precisely two numbers among the $l_j$ are
even.
\end{proof}

\begin{cor}\label{exist_quasi}
If $n\ge 3$, then $l$-quasielliptic operators exist if and only if
there is at most one odd integer among $l_1,\dots,l_n$.
\end{cor}

In the homogeneous case Theorem~\ref{exist_quasi_system} reduces
to the following result.

\begin{cor} \label{p1}
Let $l_1=\dots=l_n=l$ and let $\PjxD$ be an elliptic system of the
form~\eqref{8}. If $n\ge 2N+1$, then $l$ is even.
\end{cor}

\begin{remark}
(i) The condition $n\ge 2N+1$ of Theorem~\ref{exist_quasi_system}
is sharp. Specifically, for $n=2N$ and any $l=(l_1,\dots,l_{2N})$
consider the system of operators
\begin{equation} \label{2.2}
P_1(D):=D_1^{l_1}+iD_2^{l_2},\quad\dots,\quad
P_N(D):=D_{2N-1}^{l_{2N-1}}+ iD_{2N}^{l_{2N}}.
\end{equation}
The system~\eqref{2.2} is $l$-quasielliptic. In other words,
Theorem~\ref{exist_quasi_system} does not hold for $n\le 2N$.

(ii) For $N=1$ there is a stronger result than
Theorem~\ref{exist_quasi_system} due to Lopatinski\v{i}: for $n\ge
3$ an elliptic operator $P(D)$ is properly elliptic; in
particular, it has even order (see~\cite{Lop}-\cite{tribel}).
\end{remark}

\subsection{Characterization of $l$-quasielliptic systems by means of
a priori estimates}

We characterize $l$-quasielliptic systems with the help of a
priori estimates in the isotropic Sobolev spaces $\WplOm$. Recall
(see~\cite{BIN}, \cite{Bes}, \cite{Ber}) the following coercivity
criterion for a system $\PjxD$ in $\WplOm$,\ $p\in(1,\iy)$.

\begin{theorem} \emph{(see~\cite{BIN}, Ch.\,III, \pr\,11,
\cite{Bes} and \cite{Ber})} \label{Besov_th} Let
$l=(l_1,\dots,l_n)\in\N^n$, let $\Om$ be a domain in $\R^n$, and
let $\PjxD$ be a system of differential operators of the
form~\eqref{8} in which $a_{j\al}(\cdot)\in L^\iy(\Om)$ for
$|\al:l|\le 1$ and $a_{j\al}(\cdot)\in C(\Om)$ for $|\al:l|=1$.

Then a necessary condition for the system~\eqref{8} to be coercive
in the anisotropic Sobolev space $\WplOm$,\ $p\in(1,\iy)$, is that
it is $l$-quasielliptic; if the domain $\Om$ is bounded, then this
condition is also sufficient.
\end{theorem}

\begin{remark}
$N$-quasielliptic operators $P(x,D)$ defined in terms of the
Newton polyhedron were introduced and studied in the
book~\cite{VolGin}, Ch.\,I, \pr\,4 and Ch.\,V, \pr\,2. In
particular, an $N$-quasielliptic operator is $l$-quasielliptic if
and only if for every $x\in\Om$ the Newton polyhedron $N\(P(x)\)$
is a simplex with vertices at the origin and at the points
$(0,\dots,l_j,\dots,0)$
 (here $l_j$ is the $j$th component of the
vector). In~\cite{VolGin}, Ch,\,VI, \pr\,4, $N$-quasielliptic
operators were characterized by means of the a priori estimate
\begin{equation}\label{3.4B}
\sum_{\al\in N(P)} \|D^\al f\|_{L^2(\Om)}\le C_1
\|P(x,D)f\|_{L^2(\Om)} + C_2 \|f\|_{L^2(\Om)}, \qquad f\in
C_0^\iy(\Om),
\end{equation}
which develops the coercivity criterion in
$\overset{\circ}{W}\rule{0pt}{2mm}_2^l(\Om)$ significantly.
\end{remark}

A coercivity criterion in $\WplOmInf$ was obtained in~\cite{Mal2},
\pr\,5, Theorem 4. This result (as well as
Proposition~\ref{mal_neodn}) implies that in general an
$l$-quasielliptic system is not coercive in $\WplOm$ for
$p=1,\iy$. However, it is weakly coercive in $\WplOmInf$
(see~\cite{Mal2}, \pr\,5) and in
$\overset{\circ}{W}\rule{0pt}{2mm}_1^l(\Om)$ (see~\cite{BDLM},
\pr\,4, Theorem 3 and~\cite{dan2004}, where this was proved for
operators with constant coefficients).

\begin{theorem} \emph{(see~\cite{BIN} and \cite{Mal2})} \label{th2}
Let $\Om$ be a domain in $\R^n$, let $l=(l_1,\dots,l_n)\in\N^n$
and let $\PjxD$ be an $l$-quasielliptic system of operators of the
form~\eqref{8}. Suppose that $a_{j\al}(\cdot)\in L^{\infty}(\Om)$
for $|\al:l|\le 1$ and $a_{j\al}(\cdot)\in C(\Om)$ for
$|\al:l|=1$,\ $j\in\{1,\dots,N\}$. Then for every $\e>0$ there is
a constant $C_\e>0$ independent of $p\in(1,\iy]$ (and in the case
of operators with constant coefficients independent of
$p\in[1,\iy]$) such that the following estimate holds:
\begin{equation}\label{4.3}
\sum_{|\al:l|<1}\|D^\al f\|_{L^p(\Om)} \le
\e\sum^N_{j=1}\|P_j(x,D)f\|_{L^p(\Om)} + C_\e\|f\|_{L^p(\Om)},
\qquad f\in C_0^\iy(\Om).
\end{equation}
In particular, an $l$-quasielliptic system $\PjxD$ of the
form~\eqref{8} is weakly coercive in the space $\WplOm$ for
$p\in(1,\iy]$ (for $p\in[1,\iy]$ in the case of operators with
constant coefficients).
\end{theorem}

In the next theorem we show that for every $\e>0$ and any
$p\in(1,\iy]$ inequality~\eqref{4.3} characterizes elliptic
systems in the class of weakly coercive systems in $\WplOm$ with
constant-coefficient principal parts.

\begin{prop} \label{prop_3_8}
Let $l_1=\dots=l_n=l$, let $\Om$ be a domain in $\R^n$, and
$\PjxD$ a system of operators of the form~\eqref{8} whose
principal parts have constant coefficients, so that
$P_j^l(x,D)=P_j^l(D)$, let $a_{j\al}(\cdot)\in L^\iy(\Om)$ for
$|\al:l|< 1$,\ $j\in\{1,\dots,N\}$, and let $p\in(1,\iy]$. Then
the system of operators $\PjxD$ is elliptic if and only if the
estimate~\eqref{4.3} holds for each $\e>0$ with some constant
$C_\e>0$.

If the operators $P_j(x,D)$, $j\in\{1,\dots,N\}$, have constant
coefficients, then this criterion also holds for $p=1$.
\end{prop}

\begin{proof} The necessity follows from Theorem~\ref{th2}.

\textsl{Sufficiency.} Suppose the estimate~\eqref{4.3} holds.
Setting
$$
\wt{P}_j(x,D):=P_j(x,D)-P_j^l(D),
$$
from the triangle inequality we obtain
\begin{gather}
\notag \e'\sum \|P_j^l(D)f\|_p+C_{\e'}\|f\|_p \ge \e'\sum
\|P_j(x,D)f\|_p \\ \label{prom_1} -\e'\sum \|\td
P_j(x,D)f\|_p+C_{\e'}\|f\|_p\ ,\quad \e':=\frac{\e}{\e+1}\in(0;1).
\end{gather}
Taking into account the fact that~\eqref{4.3} holds with
$\sum_{j=1}^N \|\wt{P}_j(x,D)f\|_p$ on the left-hand side,
from~\eqref{prom_1} we obtain
\begin{gather}
\notag \e'\sum \|P_j^l(D)f\|_p+C_{\e'}\|f\|_p \ge \e'\sum
\|P_j(x,D)f\|_p-\e'\[\e'
\sum \|P_j(x,D)f\|_p+ C_{\e'}\|f\|_p\]\\
\label{prom_2} +C_{\e'}\|f\|_p=(1-\e')\[\e' \sum
\|P_j(x,D)f\|_p+C_{\e'}\|f\|_p\]\!\ge\!
(1-\e')\sum_{|\al|<l}\|D^\al f\|_p.
\end{gather}
Dividing both sides of~\eqref{prom_2} by $1-\e'>0$ and taking into
account the relation $\e'/(1-\e')=\e$ we derive the
estimate~\eqref{4.3} with $P_j^l(D)$ in place of $P_j(x,D)$:
\begin{equation}\label{eq_6}
\sum_{|\al|<l} \|D^\al f\|_p \le \e\sum_{j=1}^N
\|P_j^l(D)f\|_p+C_\e \|f\|_p\ ,\qquad f\in C^\infty_0(\Om).
\end{equation}

Let $\phi\in C_0^\iy(\Om)$,\ $\phi\not\equiv 0$. We set
$f(x):=\phi(x)e^{i t\langle\xi,x\rangle}$,\ $x\in\Om$,
in~\eqref{eq_6}. Since
$$
(D^\ga f)(x)=t^{|\ga|}\xi^\ga f(x)+t^{|\ga|-1} e^{it\langle
x,\xi\rangle} \sum_{k=1}^n \frac{\partial\xi^\ga}{\partial \xi_k}
D_k\phi(x)+o\(t^{|\ga|-1}\)
$$
by Leibniz's formula~\eqref{herm}, estimate~\eqref{eq_6} implies
the inequality
\begin{gather}
\notag \sum_{|\al|<l} \|t^{|\al|}\xi^\al f
+o(t^{l-1})\|_p-C_\e\|f\|_p \\ \label{O_t} \le \e \sum_{j=1}^N
\left\| t^l P_j^l(\xi)f+ t^{l-1}\[e^{it\langle
x,\xi\rangle}\sum_{|\al|=l} a_{j\al} \sum_{k=1}^n
\frac{\partial\xi^\al}{\partial\xi_k} D_k\phi\]+
o(t^{l-1})\right\|_p.
\end{gather}

Let $P_j^l(\xi^0)=0$,\ $j\in\{1,\dots,N\}$, for some
$\xi^0=(\xi_1^0,\dots,\xi_n^0)\in\R^n\setminus\{0\}$. Then
$\xi_s^0\neq 0$ for some $s\in\{1,\dots,n\}$. Setting $\xi=\xi^0$,
from~\eqref{O_t} we obtain
\begin{equation} \label{O_tt}
t^{l-1}|\xi_s^0|^{l-1}\|\phi\|_p+o(t^{l-1})\le \e
Ct^{l-1}|\xi^0|^{l-1}\|\phi\|_{W_p^1}+o(t^{l-1})\quad\text{as}\quad
t\to +\iy.
\end{equation}
Choosing $\e>0$ small enough, dividing both sides of~\eqref{O_tt}
by $t^{l-1}$ and passing to the limit as $t\to\iy$ we arrive at a
contradiction. Hence the system $\PjxD$ is elliptic.
\end{proof}

\subsection{Systems of principal type}

Let $\Om$ be a domain in $\R^n$ and $p\in[1,\iy]$. We denote by
$L^0_{p,\Om}(P_1,\dots,P_N)$ the space of all differential
operators $Q(x,D)$ subordinated to a system $\PjxD$, that is,
satisfying the estimate~\eqref{200}.
Following~\cite{Horm},Ch,\.\,II, \pr\,2.7, we recall the
definition.

\begin{definition}
A system of differential operators $\PjxD$ of the form~\eqref{8},
with $L^\iy(\Om)$-coefficients is called a \emph{system of
principal type in $L^p(\Om)$} if $\PjxD$ has the same force as an
arbitrary system $\{R_j(x,D)\}_1^N$ with $L^\iy(\Om)$-coefficients
and the same principal part, that is, the spaces
$L^0_{p,\Om}(P_1,\dots,P_N)$ and $L^0_{p,\Om}(R_1,\dots,R_N)$
coincide.
\end{definition}

The definition of systems of principal type readily implies the
following simple result.

\begin{prop} \label{gl_tip_coerc}
A system $\PjxD$ of differential operators of principal type in
$L^p(\Om)$ is weakly coercive in $\WplOm$,\ $p\in[1,\iy]$.
\end{prop}

\begin{proof}
Let $D^\al$ be a differential monomial, $|\al:l|<1$, and let
$\wt{P_1}:=P_1+D^\al$. Since
\begin{gather*}
P_1\in L^0_{p,\Om}(P_1,\dots,P_N),\qquad \wt{P_1}\in
L^0_{p,\Om}(\wt{P_1},\dots,P_N),\\ L^0_{p,\Om}(P_1,\dots,P_N)=
L^0_{p,\Om}(\wt{P_1},\dots,P_N)
\end{gather*}
by assumption, we have
$$
 \wt{P_1}-P_1=D^\al\in L^0_{p,\Om}(P_1,\dots,P_N).
$$
\end{proof}

By Theorem~\ref{th2} the 'force' of an elliptic system in
$L^p(\Om)$,\ $p\in[1,\iy]$, does not change under perturbations by
operators of smaller order. It turns out that this property
singles out the elliptic systems in $L^p(\R^n)$ among the totality
of weakly coercive systems in $\WplRn$ that have principal parts
with constant coefficients.

\begin{prop}\label{har_ellipt}
Let $\PjxD$ be a system of order $l$ with $L^\iy(\Om)$
coefficients such that the principal parts of the operators
$P_j(x,D)$ have constant coefficients: $P^l_j(x,D)\equiv
P^l_j(D)$.

Then the system $\PjxD$ is elliptic if and only if it is a system
of principal type in $L^p(\R^n)$ for $p\in(1,\iy]$ (or for
$p\in[1,\iy]$ in the case of constant coefficients).
\end{prop}

\begin{proof}
The necessity follows from Theorem~\ref{th2}.

\textsl{Sufficiency.} Suppose that the system $\PjxD$ is of
principal type, but not elliptic. We may assume without loss of
generality that $P_j^l(\xi^0)=0$,\ $j\in\{1,\dots,N\}$, for
$\xi^0=(1,0,\dots,0)$. Since the system $\PjxD$ is weakly coercive
in $\WplRn$ (Proposition~\ref{gl_tip_coerc}), it follows that
$D_1\in L^0_{p,\R^n}(P_1,\dots,P_N)
=L^0_{p,\R^n}(P_1^l,\dots,P_N^l)$. By
Proposition~\ref{prop_alg_ineq}, this yields the inequality
$$
 |\xi_1|\le C_1 \sum_{j=1}^N |P_j^l(\xi)|+C_2,\qquad \xi\in\R^n,
$$
which fails for $\xi=\xi^0 t$ and large $t>0$. Hence the system
$\PjxD$ is elliptic.
\end{proof}

The following assertion is a simple generalization of
H\"ormander's result in~\cite{Horm}, Ch.\,II, \pr\,2.7, Theorem
2.3 to the case $N>1$.

\begin{prop} \label{har_gl_tip}
Let $\Om$ be a bounded domain in $\R^n$. A system of differential
polynomials $\Pj$ of order $l$ is a system of principal type in
$L^2(\Om)$ if and only if
$$
  \(\nabla P_1^l,\dots,\nabla P_N^l\)(\xi)\neq 0, \qquad
  \xi\in\R^n\setminus\{0\}.
$$
\end{prop}

The proof of Proposition~\ref{har_gl_tip} is similar to
H\"ormander's (see~\cite{Horm}). However, the analogue of
H\"ormander's theorem for the system $\Pj$ must be used in place
of the theorem itself (see~\cite{Mar63}, \cite{Mal1} and
also~\cite{BIN}, Ch.\,III, \pr\,11).

\begin{remark}
(i) It seems that, as well as Propositions~\ref{prop_3_8}
and~\ref{har_ellipt}, Theorem~\ref{th2} remains true for all
$p\in[1,\iy]$ in the case of operators with variable coefficients.

(ii) In the case of $L^2(\Om)$, where $\Om$ is bounded,
Proposition~\ref{gl_tip_coerc} also follows from
Proposition~\ref{har_gl_tip} (see~\cite{Horm}).

(iii) In~\cite{VolGin} operators $P(x,D)$ of $N$-principal type in
$L^2(\Om)$ defined in terms of the Newton polyhedron $N(P)$ were
introduced and investigated. Estimates of type~\eqref{3.4B} were
obtained for them in~\cite{VolGin}, Chs. V and VI, such that the
sum on the left-hand side extends only to the interior points of
the Newton polyhedron $N(P)$. This is a significant improvement of
H\"ormander's result~\cite{Horm} mentioned above.
\end{remark}

\subsection{On the force of the tensor product
of elliptic operators in $L^\iy$}

It follows from Proposition~\ref{mal_neodn} and Theorem~\ref{th2}
that if $P(D)$ is an elliptic operator of order $l$, then $P^l(D)$
and the monomials $\{D^\al\}_{|\al|<l}$ form a basis of the space
$L_{\iy,\R^n}^0(P)$. Here we describe the structure of the space
$L_{\iy,\R^n}^0(P)$ for the operator $P(D)=P_1(D')\otimes
P_2(D'')$, where $P_1$ and $P_2$ are elliptic operators acting
with respect to different variables.

\begin{prop} \label{poln_simv}
Let $P_1(\xi)$ and $P_2(\eta)$ be elliptic polynomials of degrees
$l$ and $m$, respectively. Let
$\xi=(\xi_1,\dots,\xi_{p_1})\in\R^{p_1}$,\
$\eta=(\eta_1,\dots,\eta_{p_2})\in\R^{p_2}$,\ $p_1+p_2=n$, and
assume that $P_1(\xi)\neq 0$,\ $P_2(\eta)\neq 0$ for all
$\xi\in\R^{p_1}$,\ $\eta\in\R^{p_2}$. Then
$$
  L_{\iy,\R^n}^0(P_1P_2)=\Span\left\{Q_1 Q_2:\; Q_k\in
  L_{\iy,\R^{p_k}}^0(P_k),\; k=1,2\right\}.
$$
\end{prop}

\begin{proof}
(i) Let $Q\in\Span\{Q_1 Q_2:\; Q_k\in L_{\iy,\R^{p_k}}^0(P_k),\;
k=1,2\}$, that is,
\begin{equation} \label{9a}
Q(\xi,\eta)=\sum_{j=1}^s Q_{j1}(\xi)Q_{j2}(\eta),
\end{equation}
where $Q_{jk}\in L_{\iy,\R^{p_k}}^0(P_k)$,\;
$j\in\{1,\dots,s\}$,\; $k=1,2$. By Proposition~\ref{Q=MP+N} the
symbols $Q_{j1}(\xi)$ and $Q_{j2}(\eta)$ satisfy the relations
\begin{equation} \label{9}
 Q_{jk}=M_{jk}P_k+N_{jk}=
 \(M_{jk}+P_k^{-1} N_{jk}\)P_k=:\td M_{jk}P_k,
\end{equation}
where $M_{jk}, N_{jk}\in\cm_1(\R^{p_k})$,\; $j\in\{1,\dots,s\}$,\;
$k=1,2$. Since $P_1$ and $P_2$ are elliptic operators with
non-degenerate full symbols, it follows that
$1/P_1(\xi)\in\cm_1(\R^{p_1})$ and $1/P_2(\eta)\in\cm_1(\R^{p_2})$
(see~\cite{BDLM}, \pr\,4, Theorem 3). Then $\td
M_{jk}(\cdot)\in\cm_1(\R^{p_k})\subset\cm_1(\R^n)$,\ $k=1,2$.
Therefore, combining~\eqref{9a},~\eqref{9} and
Proposition~\ref{Q=MP+N} we arrive at $Q\in L_{\iy,\R^n}(P_1P_2)$.

(ii) Conversely, assume that $Q\in L_{\iy,\R^n}(P_1 P_2)$. We
represent the symbol $Q(\xi,\eta)$ as a sum~\eqref{9a}. We will
show that (possibly, after some rearrangement of the terms
in~\eqref{9a}) $Q_{jk}\in
L_{\iy,\R^{p_k}}^0(P_k)$,\;$j\in\{1,\dots,s\}$,\; $k=1,2$.

Let $\max_{j\in\{1,\dots,s\}}\deg Q_{j1}=l'$. First, we prove that
$l'\le l$. Indeed, without loss of generality we may assume that
$\deg Q_{j1}=l'$ for $j\in\{1,\dots,s'\}$,\; $s'\le s$, and $\deg
Q_{j1}<l'$ for $j>s'$. Collecting similar terms in the
sum~\eqref{9a} if necessary we may treat the polynomials
$Q^{l'}_{j1}(\xi)$,\ $j\in\{1,\dots,s'\}$, as linearly
independent. Choose a vector $\eta^0$ such that at least one of
the polynomials $Q_{j2}(\eta)$,\ $j\in\{1,\dots,s'\}$, does not
vanish. If we suppose that $l'>l$, then setting $\eta:=\eta^0$ in
the inequality
\begin{equation} \label{9b}
|Q(\xi,\eta)|\le C_1 |P_1(\xi)P_2(\eta)|+C_2,\qquad
\xi\in\R^{p_1},\quad \eta\in\R^{p_2},
\end{equation}
which follows from the inclusion $Q\in L_{\iy,\R^n}(P_1 P_2)$ (see
Proposition~\ref{prop_alg_ineq}), we arrive at a contradiction. In
fact, since the principal parts of the polynomials $Q_{j1}(\xi)$,\
$j\in\{1,\dots,s'\}$, do not cancel, we have
$$
 \deg Q(\xi,\eta^0)\ge \deg\sum_{j=1}^{s'}
  Q_{j1}(\xi)Q_{j2}(\eta^0)=l'
$$
on the left-hand side of~\eqref{9b}. On the right-hand side
of~\eqref{9b} we obtain
$$
 \deg P_1(\xi)P_2(\eta^0)=l
$$
as $P_2(\eta^0)\neq 0$. Hence $l'\le l$.

The proof of the relation $\max_{j\in\{1,\dots,s\}}\deg
Q_{j2}=:m'\le m$ is similar.

Further, in view of Proposition~\ref{mal_neodn}, the inclusion
$Q\in L_{\iy,\R^n}(P_1 P_2)$ implies the relation
\begin{equation} \label{9c}
Q^{l+m}(\xi,\eta)=\sum_{j=1}^s Q_{j1}^l(\xi) Q_{j2}^m(\eta)=c
P_1^l(\xi) P_2^m(\eta),\qquad \xi\in\R^{p_1},\quad
\eta\in\R^{p_2}.
\end{equation}
As above, we may assume that the polynomials
$\{Q_{j2}^m(\eta)\}_1^s$ are linearly independent. Then we can
find vectors $\eta^1,\dots,\eta^s\in\R^{p_2}$ such that
$\det\(Q_{j2}^m(\eta^r)\)\neq 0$,\; $j,\ r\in\{1,\dots,s\}$
(see~\cite{Naj}, Ch.\,V, \pr\,19, Lemma 3). Setting $\eta=\eta^r$
in~\eqref{9c} we solve the system we have obtained with respect to
the functions $Q_{j1}^l(\xi)$. This implies the relations
$Q_{j1}^l(\xi)=\la_j P_1^l(\xi)$,\; $j\in\{1,\dots,s\}$.

Arguing similarly we arrive at the relations $Q_{j2}^m(\eta)=\mu_j
P_2^m(\eta)$,\; $j\in\{1,\dots,s\}$.

Finally, since the $P_j$ are elliptic, in view of
Propositions~\ref{mal_neodn} and~\ref{har_ellipt}, we have
$Q_{jk}\in L_{\iy,\R^{p_k}}^0(P_k)$, that is, $Q\in\Span\{Q_1
Q_2:\; Q_k\in L_{\iy,\R^{p_k}}^0(P_k),\; k=1,2\}$.
\end{proof}

The non-degeneracy of the full symbols of the operators $P_1$ and
$P_2$ is essential for Proposition~\ref{poln_simv} to hold. The
following result shows that even in the case of the product $P_1
P_2$ of two homogeneous elliptic operators $P_1$ and $P_2$ acting
on different groups of variables, the space $L^0_{\iy,\R^n}(P_1
P_2)$ contains no differential monomials.

\begin{prop}
Let $P_1(\xi)$ and $P_2(\eta)$ be homogeneous elliptic polynomials
of degrees $l$ and $m$, respectively, and let
$\xi=(\xi_1,\dots,\xi_{p_1})\in\R^{p_1}$,
$\eta=(\eta_1,\dots,\eta_{p_2})\in\R^{p_2}$, $p_1,p_2>1$,
$p_1+p_2=n$. Then the inclusion $D^\al\in L^0_{\iy,\R^n}(P_1 P_2)$
does not hold for any $\al\neq 0$.
\end{prop}

\begin{proof}
Let
\begin{gather*}
D^\al=D'^{\al_1}D''^{\al_2}\in L_{\iy,\R^n}^0(P_1P_2),\qquad
D':=(D_1,\dots,D_{p_1}),\qquad D'':=(D_1,\dots,D_{p_2}),\\
\al_k\in\Z_+^{p_k},\qquad k=1,2.
\end{gather*}
By Proposition~\ref{Q=MP+N} the estimate
\begin{equation} \label{d1d2}
\|D'^{\al_1}D''^{\al_2}f\|_{L^\iy(\R^n)} \le C_1
\|P_1(D')P_2(D'')f\|_{L^\iy(\R^n)}+
 C_2 \|f\|_{L^\iy(\R^n)},\quad f\in C_0^\iy(\R^n),
\end{equation}
is equivalent to the relation
\begin{equation} \label{9d}
  \xi^{\al_1}\eta^{\al_2}=M(\xi,\eta)P_1(\xi)P_2(\eta)+N(\xi,\eta),\qquad
  \xi\in\R^{p_1},\quad \eta\in\R^{p_2},
\end{equation}
where $M,N\in\cm_1(\R^n)$. We set $\xi_1=\xi_2=\dots=\xi_{p_1}$
and $\eta_1=\eta_2=\dots=\eta_{p_2}$ in~\eqref{9d}. Taking into
account that the 'restriction' of a multiplier on $L^\iy(\R^n)$ to
a subspace $E\subset\R^n$ is also a multiplier on $L^\iy(E)$
(see~\cite{St}, Ch.\,IV, \pr\,7.5), from Proposition~\ref{Q=MP+N}
we see that the differential monomial with symbol
$\xi_1^{|\al_1|}\eta_1^{|\al_2|}$ can be estimated in
$L^\iy(\R^2)$ in terms of another differential monomial with
symbol $c\xi_1^l\eta_1^m$, and we have
$c=P_1(1,\dots,1)P_2(1,\dots,1)\neq 0$ because $P_1$ and $P_2$ are
elliptic polynomials. By Boman's theorem (see~\cite{Bom}, \pr,5,
Theorem 2) for $\al\neq 0$ this is possible only in the case of
$|\al_1|=l$,\ $|\al_2|=m$. In this case the monomial
$\xi^{\al_1}\eta^{\al_2}$ has degree $l+m$ and by
Proposition~\ref{mal_neodn} we obtain
$$
   \xi^{\al_1} \eta^{\al_2}=C P_1^l(\xi) P_2^m(\eta)=C P_1(\xi)
   P_2(\eta),
$$
which yields $P_1(\xi)=C_1\xi^{\al_1}$ and
$P_2(\eta)=C_2\eta^{\al_2}$. However, the polynomials
$\xi^{\al_1}$ and $\eta^{\al_2}$ are not elliptic for $p_1>1$ and
$p_2>1$. Thus, the estimate~\eqref{d1d2} does not hold for any
$\al\neq 0$.
\end{proof}

\section{Weak coercivity in the isotropic
    space $\protect\WplRn$}

Here we will study properties of weakly coercive systems of order
$l$ of the form
\begin{equation} \label{izotr_system}
 P_j(x,D)=\sum_{|\al|\le l} a_{j\al}(x) D^\al,
 \qquad j\in\{1,\dots,N\},
\end{equation}
in the isotropic Sobolev space $\WplRn$. In particular, for these
systems we obtain an analogue of Theorem~\ref{exist_quasi_system}.

\subsection{Properties of weakly coercive systems
 in $\protect\WplRn$}

\begin{prop} \label{prop_weak_coerc}
Let $\PjxD$ be a system of operators of the
form~\eqref{izotr_system} of order $l\ge 2$ with the coefficients
$a_{j\al}(\cdot)\in L^\iy_{\loc}(\R^n)$ for $|\al|\le l-1$ and
such that the coefficients of the principal parts are constant.
Assume also that the system $\PjxD$ is weakly coercive in the
isotropic space $\WplRn$,\; $p\in [1,\iy]$.

\emph{(i)} If the operators $P_j(x,D)$,\ $j\in\{1,\dots,N\}$, have
continuous coefficients, then for any fixed $x_0\in\R^n$ the zero
set
$$
\mathcal N(x_0,P):=\{\xi\in\R^n:
P_1(x_0,\xi)=\dots=P_N(x_0,\xi)=0\}
$$
of the system of polynomials $\{P_j(x_0,\xi)\}_1^N$ is compact.

\emph{(ii)} For any system $\{Q_j(x,D)\}_1^N$, where the
$Q_j(x,D)$ are operators of order $\le l-2$ with
$L^\iy(\R^n)$-coefficients, the system $\{P_j(x,D)+Q_j(x,D)\}_1^N$
is also weakly coercive in $\WplRn$.

(iii) Let $\xi^0\in\R^n\setminus\{0\}$ be a zero of the map
$P^l=\(P_1^l,\dots,P_N^l\):\R^n\to\R^{2N}$, that is,
$P_j^l(\xi^0)=0$,\; $j\in\{1,\dots,N\}$. If $n\ge 2N+1$, then the
Jacobi matrix of the map $P^l:=(P_1^l,\dots,P_N^l):\R^n\to\R^{2N}$
at the point $\xi^0$ has rank less than $2N$.

(iv) In addition, let $a_{j\al}(\cdot)\in C^1(\R^n)$ for
$|\al|=l-1$. Let $N=1$,\ $p=\iy$,\ $n\ge 2$. If
$\xi^0\in\R^n\setminus\{0\}$ is a zero of the polynomial
$P^l(\xi)$, then $\nabla P^l(\xi^0)\neq 0$. In particular, if
$n=2$, then the polynomial $P^l(\xi)$ has simple zeros.
\end{prop}

\begin{proof}
(i) Suppose that the set $\mathcal N(x_0,P)$ is not compact for
some $x_0\in\R^n$. Then for some sequence $\{\xi^{(m)}\}_1^\iy$,\
$\lim_{m\to+\iy} \xi^{(m)}=\iy$, we have $P_j(x_0,\xi^{(m)})=0$,\
$j\in\{1,\dots,N\}$ (without loss of generality we can assume that
$|\xi^{(m)}|>1$ for $m\in\N$).

Since the principal parts of the operators $P_j(x,D)$ have
constant coefficients, the symbols $P_j(x,\xi^{(m)})
=P_j(x,\xi^{(m)})- P_j(x_0,\xi^{(m)})$ have degree $\le l-1$ (with
respect to $\xi$). Then for each $\e>0$ there exists a ball
$$
B_\delta(x_0):=\{x\in\R^n: |x-x_0|\le\delta\},\qquad \delta>0,
$$
such that for $x\in B_\delta(x_0)$ we have
\begin{equation} \label{l-1}
P_j(x,\xi^{(m)})\le \frac{\e}{N} |\xi^{(m)}|^{l-1}, \qquad
\left|\frac{\partial P_j}{\partial\xi_k}(x,\xi^{(m)})\right|\le
C|\xi^{(m)}|^{l-1},
\end{equation}
where  $m\in\N$,\; $j\in\{1,\dots,N\}$,\; $k\in\{1,\dots,n\}$.

Let $\psi\in C_0^\iy\(B_\delta(x_0)\)$,\ $\psi\not\equiv 0$.
Consider the functions $f_r(x):=\psi_r(x)e^{i\langle
x,\xi^{(m)}\rangle}$, where $\psi_r(x):=\psi(x/r)$,\ $r>0$.
Then~\eqref{l-1} implies the estimates
\begin{equation} \label{P_j}
\begin{aligned}
\|P_j(x,\xi^{(m)})e^{i\langle x,\xi^{(m)}\rangle }\psi_r(x)\|_p
\le \frac{\e}{N} |\xi^{(m)}|^{l-1}\cdot\|\psi_r\|_p\ , \\
\left\|\frac{\partial P_j}{\partial\xi_k}(x,\xi^{(m)}) e^{i\langle
x,\xi^{(m)}\rangle} D_k\psi_r(x)\right\|_p \le
 C|\xi^{(m)}|^{l-1}\cdot\|D_k\psi_r\|_p\ .
 \end{aligned}
\end{equation}
Setting $f=f_r$ in inequality~\eqref{10'} and using Leibniz's
formula~\eqref{herm} we obtain
\begin{gather}
\notag \sum_{|\al|<l} \left\|\(\xi^{(m)}\)^\al
 e^{i\langle x,\xi^{(m)}\rangle}
\psi_r(x)+o\(|\xi^{(m)}|^\al\)\right\|_p\le C_1\sum_{j=1}^N
\left\|P_j(x,\xi^{(m)}) e^{i\langle x,\xi^{(m)}\rangle}\psi_r(x) \right.\\
\left.+\sum_{k=1}^n \frac{\partial P_j}{\partial\xi_k}
(x,\xi^{(m)}) e^{i\langle x,\xi^{(m)}\rangle}
D_k\psi_r(x)+o\(|\xi^{(m)}|^{l-1}\)\right\|_p+ C_2\|\psi_r\|_p\
.\label{o(xi)}
\end{gather}
Taking account of~\eqref{P_j} and the obvious inequality
$|\xi^{(m)}|^{l-1}\le C'\sum_1^n |\xi_k^{(m)}|^{l-1}$,
from~\eqref{o(xi)} we obtain
\begin{equation} \label{xi_k^m}
 |\xi^{(m)}|^{l-1}\cdot\|\psi_r\|_p
 \le C'C_1 |\xi^{(m)}|^{l-1} \[ \e\|\psi_r\|_p
 +C\sum_{k=1}^n \|D_k\psi_r\|_p\]+o\(|\xi^{(m)}|^{l-1}\).
\end{equation}
Dividing both sides of~\eqref{xi_k^m} by $|\xi^{(m)}|^{l-1}$ and
passing to the limit as $m\to\iy$ we see that
\begin{equation} \label{rnp}
 \|\psi_r\|_p \le C'C_1 \[\e \|\psi_r\|_p+C \sum_{k=1}^n
 \|D_k\psi_r\|_p\].
\end{equation}
Now let $p\in[1,\iy)$. Note that
$$
\|\psi_r\|_p=r^{n/p}\|\psi\|_p,\qquad \|D_k\psi_r\|_p=r^{-1}\cdot
r^{n/p}\|D_k\psi\|_p.
$$
Cancelling out the factor $r^{n/p}$ in~\eqref{rnp} and letting
$r\to\iy$ we arrive at the estimate $\|\psi\|_p \le \e
C'C_1\|\psi\|_p$. (For $p=\iy$ we must set here
$r^{n/p}=r^{n/\iy}=1$.) Finally, taking $\e>0$ sufficiently small
we arrive at a contradiction.

(ii) The embedding theorems (see~\cite{BIN}, Ch.\,III, \pr\,9)
imply that for any $\e>0$ there is $C_\e>0$ such that
\begin{equation}\label{eq_4}
\sum_{j=1}^N\|Q_j(x,D)f\|_p \le \e \sum_{|\al|\le l-1} \|D^\al
f\|_p +C_\e \|f\|_p,\qquad f\in C_0^\infty(\R^n).
\end{equation}

Combining~\eqref{eq_4} and~\eqref{10'} yields
\begin{equation} \label{2''}
\sum_{j=1}^N \|P_j(x,D)f+Q_j(x,D)f\|_p+\|f\|_p
 \ge \(C_1-\e\) \sum_{|\al|\le l-1} \|D^\al f\|_p-C_\e
 \|f\|_p
 \end{equation}
for all $f\in C_0^\infty(\R^n)$. We choose $\e<C_1$
in~\eqref{eq_4}. Then~\eqref{2''} implies the following estimate:
$$
 \sum_{|\al|\le l-1} \|D^\al f\|_p\le \(C_1-\e\)^{-1}
 \sum_{j=1}^N \|P_j(x,D)f+Q_j(x,D)f\|_p+\(C_\e+1\)\|f\|_p\ .
$$
The weak coercivity of the system $\{P_j(x,D)+Q_j(x,D)\}_1^N$
follows from this inequality.

(iii)  First let $n=2N+1$. By (ii) we may assume that the
operators  $P_j$ have the form $P_j(x,D)=P_j^l(D) +
P_j^{l-1}(x,D)$, where  $P_j^{l-1}(x,D):=\sum_{|\al|=l-1}
a_{j\al}(x)D^\al$,\ $j\in\{1,\dots,N\}$. Assume the contrary: the
rank of the Jacobi matrix of the map
$P^l=(P_1^l,\dots,P_N^l):\R^{2N+1}\to\R^{2N}$ at the point $\xi^0$
equals  $2N$:
\begin{equation} \label{J_P} \rank
\frac{\partial (\Re\ P_1^l,\ \Im\ P_1^l,\ \dots,\ \Re\ P_N^l,\
\Im\ P_N^l)}{\partial(\xi_1,\dots,\xi_{2N+1})}(\xi^0)=2N.
\end{equation}

Consider a smooth parametrization of the sphere $\S^{2N}$ in a
neighborhood of the point  $\xi^0$, that is, a diffeomorphism
$\Phi:=(\Phi_1,\dots,\Phi_{2N+1}): B_\e^{2N}\to V$, where
$B_\e^{2N}:=\{\phi\in\R^{2N}:|\phi|<\e\}$,\ $V$  is an open
neighborhood of the point  $\xi^0\in \S^{2N}$, and
$\Phi(0)=\xi^0$. Since $\Phi$ is the diffeomorphism, the Jacobi
matrix of $\Phi$ at the origin has rank $2N$:
\begin{equation} \label{ogr_1}
\rank \frac{\partial\(\Phi_1,\dots,\Phi_{2N+1}\)}{\partial
\(\phi_1,\dots,\phi_{2N}\)}(0)=2N.
\end{equation}

Let  $\td T:=P^l\circ\Phi$ be the composition of the maps $P^l$
and $\Phi$,\; $\td T:B_\e^{2N}\to\R^{2N}$. Since $\Phi(0)=\xi^0$
and since by assumption $P_j^l(\xi^0)=0$,\; $j\in\{1,\dots,N\}$,
it follows that
\begin{equation} \label{T(0)}
 \td T(0)=0.
\end{equation}
In addition, it follows from~\eqref{J_P} and~\eqref{ogr_1} that
the map $\td T$ has Jacobian $J_{\td T}$ distinct from zero at the
origin:
\begin{equation} \label{rank_J}
      J_{\td T}(0)\neq 0.
\end{equation}

By~\eqref{T(0)} and~\eqref{rank_J} the origin is an isolated zero
of the map $\td T:B_\e^{2N}\to\R^{2N}$. Let $U\subset\R^{2N}$ be
an open neighborhood of the origin such that $\td T(\phi)\not=0$,\
$\phi\in \ol U\setminus\{0\}$, and  $J_{\td T}(\phi)\not =0$,\;
$\phi\in\ol U$. We can assume for simplicity that  $U=B_\e^{2N}$.
We denote by $\Phi^r$ the map from $U$ to $\S_r^{2N}$ defined by
$\Phi^r(\phi): =r\Phi(\phi): =(r\Phi_1(\phi), \dots,
r\Phi_{2N+1}(\phi))$ for $r>0$, and we denote by $\td T^r:=
P^l\circ\Phi^r:$,\ $\td T^r: U\to\R^{2N}$, the composition of the
maps $P^l$ and $\Phi^r$. Since the components of $P^l$ are
homogeneous polynomials of degree  $l$, we have $\td
T^r=P^l(\Phi^r)=P^l(r\Phi)=r^l P^l(\Phi)=r^l \td T$. Similarly,
$J_{\td T^r}=r^{2N(l-1)} J_{\td T}$. It follows that the maps $\td
T^r$  satisfy the same relations as $\td T$, that is,
\begin{equation}\label{4.11}
\td T^r(\phi)\neq 0, \quad  \phi\in \ol U\setminus\{0\},\qquad
J_{\td T^r}(\phi)\not =0, \quad \phi\in\ol U.
\end{equation}
By~\eqref{4.11}, for each $r>0$ the vector field $\td T^r$ does
not vanish on the boundary $\partial U = \S_\e^{2N-1}$ and it has
only one singular point in the interior of $U$.

We may assume without loss of generality that $J_{\td T}(0)>0$
in~\eqref{rank_J}, and hence $J_{\td T^r}(\phi)>0$
in~\eqref{4.11}. Therefore, the singular point $0$ of the vector
field $\td T^r$ has index $1$,\ $\ind(0,\td T^r)$=1, where, as
usual, $\ind(x_0, F)$ is the index of the singular point $x_0$ of
the vector field  $F$. Since the rotation number $\ga(\td T^r,
\partial U)$ of the vector field  $\td T^r$ on
$\partial U$ is equal to the sum of the indices of singular points
of the vector field $\td T^r$ in the interior of $U$, it follows
that
\begin{equation}\label{4.12}
\ga(\td T^r, \partial U)= \ind(0, \td T^r)= \sign J_{\td
T^r}(0)=1.
\end{equation}
We fix  $x_0\in\R^n$.  Since  $P^l(\xi)\not =0$,\ $\xi\in
\partial V$, there exists  $r_0>0$ such that for $r>r_0$ we have
\begin{equation} \label{P_r}
P(x_0,r\xi)= P^l(r\xi)+ P^{l-1}(x_0,r\xi) = r^l[P^l(\xi)+
r^{l-1}P^{l-1}(x_0,\xi)]\neq 0, \qquad \xi\in \partial V.
\end{equation}
Let us introduce the maps  $\td P:=P\circ\Phi:U\to\R^{2N}$, where
$$
P:=(\Re\ P_1,\Im\ P_1,\dots,\Re\ P_N,\Im\
P_N):\R^{2N+1}\to\R^{2N},
$$
and $\td P^r:=P\circ\Phi^r:U\to\R^{2N}.$  Taking account
of~\eqref{P_r} and the fact that  $\Phi$ is a diffeomorphism, we
obtain $\td P^r(\phi)\neq 0$ for $\phi\in\partial U$. Hence, for
each $r>r_0$, the maps $\td T^r:\partial
U\to\R^{2N}\setminus\{0\}$ and $\td P^r:
\partial U\to\R^{2N}\setminus\{0\}$ are homotopic
in the space of continuous maps from $\partial U$ into
$\R^{2N}\setminus\{0\}$, and the homotopy is given by
$$
\Psi_r(t,\xi)= tP(x_0,\xi)+(1-t) P^l(\xi) = P^l(\xi)+ t
P^{l-1}(x_0,\xi):\ \partial V \to\R^{2N+1}\setminus\{0\},  \quad
t\in[0,1].
$$
But homotopic fields have equal rotations. Hence, taking account
of~\eqref{4.12} we obtain
\begin{equation}
\ga(\td P^r, \partial U) = \ga(\td T^r, \partial U)=1.
\end{equation}
Thus, the map $P^r: \partial U \to \R^{2N}\setminus\{0\}$ is
homotopically non-trivial, hence by Theorem~\ref{step_otobr} any
continuous extension of it into the interior of $U$ has zeros for
each $r>r_0$. In particular, for every $r>r_0$ there exists
$\phi^0(r)\in U$ such that $\td
P^r(\phi^0(r))=P(\Phi^r(\phi^0(r)))=0$, where
$\Phi^r(\phi^0(r))\in\S_r^{2N}$. This contradicts assertion (i)
that the zero set is compact. Thus, the statement is proved for
$n=2N+1$.

Now assume that $n>2N+1$, while the rank of the Jacobi matrix of
$P^l$ at $\xi^0$ is $2N$. We choose  $2N$ columns containing a
non-trivial minor and set the remaining $n-2N$ columns equal to
zero. By (i) the zero set $\mathcal N(x_0,P)$ is compact since the
system~\eqref{izotr_system} is weakly coercive in $\WplRn$. This
property still holds if we restrict the polynomials to a subspace,
hence the proof reduces to the previous case of $n=2N+1$.

(iv)  The proof is based on Proposition~\ref{mal_neodn}. Assume
the contrary, that is, suppose that the operator  $P(x,D)$ is
weakly coercive in $\WplRnInf$, while
\begin{equation} \label{iv}
\nabla P^l(\xi^0)=\(\frac{\partial P^l}{\partial\xi_1}, \dots,
\frac{\partial P^l}{\partial\xi_n}\)(\xi^0)=0.
\end{equation}
After a suitable orthogonal change of the variables
$\xi_1,\dots,\xi_n$, we may assume that $\xi^0=(0,\dots,0,1)$.
Then equality~\eqref{iv} means that the coefficients of the
monomials $\xi_n^l$ and $\xi_n^{l-1}\xi_j$,\
$j\in\{1,\dots,n-1\}$, in $P^l(\xi)$ are zero. Consider the
smallest $k\in\N$ such that at least one of the monomials
$\xi^\al$,\ $|\al|=l$,\ $\al_n=l-k$, occurs in $P^l$ with a
non-zero coefficient (such a $k$ exists since $P^l\not\equiv 0$
and $k\ge 2$ by~\eqref{iv}). Let
\begin{equation}\label{vec}
l':=(l'_1,\dots,l'_n):=\(\frac{k(l-1)}{k-1},\ \dots,\
\frac{k(l-1)}{k-1},\ l-1\).
\end{equation}
The vector $l'$ defines a hyperplane  $\pi':\ |\al:l'| =
\sum_{j=1}^n \xi_j/l'_j=1$, or
\begin{equation} \label{hyperpl}
\pi':\  (k-1)(\xi_1+\dots+\xi_{n-1})+k\xi_n=k(l-1).
\end{equation}
Let $P(x,\xi) = \sum_{|\al|\le l} a_\al(x) \xi^\al$  be the full
symbol of the operator $P(x,D)$. Clearly, $(0,\dots,0,l-1)\in
\pi'$, and $\alpha\in\pi'$ for $|\alpha|=l$ and $\alpha_n=l-k$. We
claim that the exponents $\alpha$  of the other monomials
$\xi^{\alpha},\ |\alpha|\le l$, lie 'below' the hyperplane $\pi'$.
If $|\al|=l$ and $\al_n<l-k$, then
\begin{equation} \label{ineq0}
(k-1)(\al_1+\dots+\al_{n-1})+k\al_n<k(l-1).
\end{equation}
Finally, if either $|\alpha|\le l-1$ or $|\alpha|=l-1$, but
$\alpha_n<l-1$, then inequality~\eqref{ineq0} also holds. Thus,
the exponents of all  monomials $\xi^{\alpha}$ either lie 'below'
the hyperplane $\pi'$ or belong to it. Therefore, the
$l'$-principal form $P^{l'}(x,\xi):=\sum_{|\al:l'|= 1} a_\al(x)
\xi^\al$ of the full symbol $P(x,\xi)$ has the form
\begin{equation} \label{usl}
P^{l'}(x,\xi)=c_0(x) \xi_n^{l-1}+\sum_{|\al|=l,\ \al_n=l-k} a_\al
\xi^\al,\qquad c_0(x):=a_{0,\dots,0,l-1}(x).
\end{equation}

Now we apply Proposition~\ref{mal_neodn} to the operators
$Q(D):=D_n^{l-1}$,\ $P(x,D)$ and the vector $l'$ of the
form~\eqref{vec}. Clearly, $Q^{l'}(D)=Q(D)=D_n^{l-1}$, and taking
account of~\eqref{usl},
\begin{equation} \label{tozh}
\xi_n^{l-1}\equiv \la(x)\[c_0(x) \xi_n^{l-1}+\sum_{|\al|=l,\
\al_n=l-k} a_\al \xi^\al\].
\end{equation}
From~\eqref{tozh} we obtain
$$
\la(x) c_0(x)\equiv 1,\quad \la(x) a_\al\equiv 0,\qquad
x\in\R^n,\quad |\al|=l,\quad \al_n=l-k.
$$
Hence $a_\al=0$,\ $|\al|=l$,\ $\al_n=l-k$. This contradicts the
choice of $k$. Thus, $\nabla P^l(\xi^0)\!\neq\!0$.
\end{proof}

\begin{remark} \label{3a}
(i) In the case of a weakly coercive system $\Pj$ with constant
coefficients the compactness of the zero set of the map
$P=(P_1,\dots,P_N):\R^n\to\R^{2N}$  follows from the algebraic
inequality~\eqref{alg_ineq}.

(ii) The condition $n\ge 2N+1$ in assertion (iii) is sharp. For
instance, the Jacobi matrix of the system
$\{(\xi_1+i)(\xi_2+i)$,\; $(\xi_3+i)(\xi_4+i)\}$ has rank one at
the point $(1,0,0,0)$ and rank two at $(1,0,1,0)$.

(iii) In the case of constant coefficients assertion (iv) has
significance only for $n=2$ since, in view of
Theorem~\ref{deLeu-Mir_intr}, any weakly coercive operator in
$\WplRnInf$ is elliptic for $n\ge 3$.
\end{remark}

\subsection{For which $l$ do weakly coercive systems exist?}

In the next theorem we extend Theorem~\ref{exist_quasi_system} to
systems of operators with constant coefficients, that are weakly
coercive in $\WplRn$,\ $p\in[1,\iy]$.

\begin{theorem} \label{prop6}
Let $\Pj$ be a system  of order $l$ that is weakly coercive in the
isotropic Sobolev space $\WplRn$,\ $p\in[1,\iy]$, and suppose that
$n\ge 2N+1$. If the map
$$
 P^l:=(P_1^l,\dots,P_N^l):\R^n\to\R^{2N}
$$
has finitely many zeros on the sphere  $\S^{n-1}$, then $l$ is
even.
\end{theorem}

\begin{proof}
(i) Let $n=2N+1$.  Since the map $P^l$ has finitely many zeros on
$\S^{2N}$, there exists a unit sphere $\S^{2N-1}$ such that the
restriction $P^l\lceil\S^{2N-1}$ has no zeros. Here the sign
$\lceil$ denotes the restriction of a map to the corresponding
set. Since all the polynomials $P^l_j(\xi)$ are homogeneous, we
can assume without loss of generality that
$\S^{2N-1}:=\{x\in\S^{2N}:\ x_n=0\}$.

As in~\eqref{T},  we denote by $T=(T_1,\dots,T_{2N}):
\S^{2N-1}\to\R^{2N}$ the 'restriction' of the map $P^l$ to the
sphere $\S^{2N-1}$, that is,
 \begin{equation} \label{TT}
 T_{2j-1}(\xi):=\Re\
P_j^l(\xi)\lceil\S^{2N-1},\qquad
 T_{2j}(\xi):=\Im\ P_j^l(\xi)\lceil\S^{2N-1},\qquad j\in\{1,\dots,N\}.
\end{equation}
Since  $P^l\lceil\S_r^{2N-1}\neq 0$ for all $r>0$, the map
$$
T^r:= \frac{\(T_1,\dots,T_{2N}\)}{\|P^l\|}:
   \ \S_r^{2N-1}\to \S_r^{2N-1}
$$
is continuous. If $l$ is odd, then $P^l$ is odd:
$P^l(-\xi)=-P^l(\xi)$. Then by Theorem~\ref{borsuk_2} the maps
$T^r$ have odd degree $\deg T^r=2k+1$, and hence are homotopically
nontrivial (see~\cite{spenyer}).

Consider the restriction of the map $P=(P_1,\dots,P_N)$ to the
sphere $\S^{2N-1}_r$. We denote
$$
R_{2j-1}(\xi):=\Re\ P_j(\xi)\lceil\S_r^{2N-1},\qquad
R_{2j}(\xi):=\Im\ P_j(\xi)\lceil\S_r^{2N-1},\qquad
j\in\{1,\dots,N\}.
$$
For sufficiently large $r$ the maps
$$
T^r:\ \S^{2N-1}_r\to\R^{2N}\setminus\{0\}, \qquad
R^r:=(R_1,\dots,R_{2N}):\ \S^{2N-1}_r\to\R^{2N}\setminus\{0\}
$$
are homotopic in the space of continuous maps from $\S^{2N-1}_r$
into $\R^{2N}\setminus\{0\}$. Indeed, for $\xi=r\eta$,\;\
$\eta\in\S^1$, and large $r>0$ we have
\begin{equation}\label{Rnonzero}
R^r(\xi)= R^r(r\eta)=  r^l T^1(\eta) + O(r^{l-1}) = r^lT^1(\eta)(
1 + O(r^{-1})) \neq 0.
\end{equation}
Therefore, the maps  $R^r$ and  $T^r$ are homotopic in
$\R^{2N}\setminus\{0\}$ since by~\eqref{Rnonzero} the homotopy
$tR^r+(1-t)T^r$ does not vanish for large $r>0$:
$$tR^r+(1-t)T^r:\
\S_r^{2N-1}\to\R^{2N}\setminus\{0\}.
$$
Hence the map $R^r$ has the same degree as $T^r$, $\deg R^r=2k+1$,
and is also homotopically non-trivial. Thus, by
Theorem~\ref{step_otobr} any continuous extension of it into the
interior of the closed ball $B_r^{2N}$ has a zero. In particular,
the map
\begin{equation} \label{otobr_sfera}
 \wt{P^r}(\xi')=\wt{P^r}(\xi_1,\dots,\xi_{n-1}):=
P\(\xi_1,\dots,\xi_{n-1},\sqrt{r^2-|\xi'|^2}\):
B_r^{2N}\to\R^{2N},
\end{equation}
where $\xi':=(\xi_1,\dots,\xi_{n-1})$, which is a continuous
extension of $R^r: \S_r^{2N-1}\to\R^{2N}\setminus\{0\}$ into
$B_r^{2N}$, also has zeros.  Since the hemisphere
$\S^{2N}_{+r}:=\{x\in\S_r^{2N}:\ x_n\ge 0\}$ is homeomorphic to
the ball  $B^{2N}_r$,  the maps~\eqref{otobr_sfera} define in a
natural way maps $P\lceil\S_{+r}^{2N}=R^r\lceil\S_{+r}^{2N}$ of
the hemisphere $\S^{2N}_{+r}$ into $\R^{2N}$. Thus, the maps
$P\lceil\S_{+r}^{2N}$ have zeros for large $r$. This contradicts
Proposition \ref{prop_weak_coerc}, (i).

(ii) Now suppose $n>2N+1$. Setting  $\xi_k=0$ for
$k\in\{2N+2,\dots,n\}$ we consider the 'restricted' system
$\{\wt{P_j}(\xi_1,\dots,\xi_{2N+1},0,\dots,0)\}_1^N$. The map
$\wt{P^l}:=(\wt{P_1^l},\dots,\wt{P_N^l}):\R^{2N+1}\to\R^{2N}$ also
has a finite number of zeros on the sphere $\S^{2N}$. Moreover, by
Proposition~\ref{prop_weak_coerc}, (i) the zero set $N(P)$ of the
symbols of the system $\{P_j(\xi)\}_1^N$ is compact since the
system $\Pj$ is weakly coercive in $\WplRn$. The zero set
$N(\wt{P})$ of the restricted system $\{\wt{P_j}(\xi)\}_1^N$
remains compact in $\R^{2N+1}$. To complete the proof it remains
to repeat the reasoning in item (i) for the system
$\{\wt{P_j}(\xi)\}_1^N$.
\end{proof}

\begin{remark}
(i) We conjecture that the conclusion of Theorem~\ref{prop6} that
$l$ is even holds without additional assumptions on the system
$\Pj$.

(ii) The condition $n\ge 2N+1$ is essential in
Theorem~\ref{prop6}. For instance, if $n=2N$ then the system
$$
  P_1(D):=(D_1+iD_2)^l,\quad
  P_2(D):=(D_3+iD_4)^l,\quad\dots,\quad
  P_N(D):=(D_{n-1}+iD_n)^l
$$
is elliptic for any  $l$.

(iii) It is clear from the proof of Theorem~\ref{prop6} that the
condition that the map $P^l$ has only finitely many zeros on the
sphere $\S^{n-1}$ can be relaxed, instead only assuming that there
exists a sphere $\S^{n-2}$ free of zeros of $P^l$. However,
examples do exist where the latter condition is not fulfilled. For
instance, if $N=2$ and $n=5$, then the system
$$
  P_1(D):=(D_1+i)(D_2+i),\qquad P_2(D)=D_3^2+D_4^2+D_5^2
$$
is weakly coercive in
$\overset{\circ}{W}\rule{0pt}{2mm}_p^l(\R^5)$,\ $p\in[1,\iy]$,
although the restriction of the map $P^l=(\xi_1\xi_2,\;
\xi_3^2+\xi_4^2+\xi_5^2)$ to any sphere $\S^3$ has a zero.
\end{remark}

\subsection{A characterization of weakly coercive systems of operators with
constant coefficients in $\protect\WplRnInf$}

Here we obtain an analogue of Theorem~\ref{deLeu-Mir_intr} for the
case of a homogeneous system of operators with constant
coefficients. To this end we will use the procedure, described in
the following proposition, of 'restricting' an estimate to a
subspace.

\begin{prop}\label{suzh}
Let $Q(x,D)$ and $\PjxD$ be operators of the form~\eqref{Q_Pj}
with $L^\iy(\Om)$ coefficients and let
$(D',0):=(D_1,\dots,D_m,0,\dots,0)$. Then for any $m<n$ the
estimate~\eqref{200} with $p=\iy$ and $\Om=\R^n$ implies the
'restricted' estimate
\begin{equation} \label{suzh_est}
 \|Q(x,D',0)\td f\|_{L^\iy(\R^n)} \le C_1 \sum_{j=1}^N
\|P_j(x,D',0)\td f\|_{L^\iy(\R^n)}+C_2\|\td f\|_{L^\iy(\R^n)}
 \end{equation}
for all $\td f\in C_0^\iy(\R^m)$.  Moreover, if the operators $Q$
and $P_j$ have constant coefficients, then estimate~\eqref{200}
remains valid if all the operators are restricted to an arbitrary
subspace $E\subset\R^n$.
\end{prop}

\begin{proof}
Let $\phi\in C_0^\iy(\R^{n-m})$ be a 'cutoff' function equal to
$1$ in a neighborhood of the origin. Consider functions $f\in
C_0^\iy(\R^n)$ of the following form:
\begin{equation} \label{new_prob}
f(x_1,\dots,x_n):= \td f(x_1,\dots,x_m) \phi(x_{m+1},\dots,
x_n),\qquad\text{where}\quad \td f\in C_0^\iy(\R^m).
\end{equation}
Further, for any $r>0$ and any function $f$ of the
form~\eqref{new_prob} we denote by $f_r$ the function
\begin{equation} \label{new_prob_r}
f_r(x):=f\(x_1,\dots,x_m,\frac{x_{m+1}}{r},\dots,\frac{x_n}{r}\)=
\td f(x_1,\dots,x_m) \phi\(\frac{x_{m+1}}{r},\dots,
\frac{x_n}{r}\).
\end{equation}
We substitute~\eqref{new_prob_r} into~\eqref{200}. For any
differential monomial $D^\al=D_1^{\al_1}\dots D_n^{\al_n}$ we have
$$
 D^\al f_r=r^{-(\al_{m+1}+\dots+\al_n)}(D^\al f)_r,
$$
hence in view of the estimates
$$
\|a_{j\al}(x) D^\al f_r\|_{L^\iy(\R^n)} \le
\|a_{j\al}(x)\|_{L^\iy(\R^n)} \|D^\al f_r\|_{L^\iy(\R^n)} \le C
r^{-(\al_{m+1}+\dots+\al_n)},
$$
passing to the limit as $r\to +\iy$ in the inequality obtained we
arrive at~\eqref{suzh_est}.

If the operators $Q$ and $P_j$ have constant coefficients, then
every $L^\iy(\R^n)$-norm in~\eqref{suzh_est} is equal to the
corresponding $L^\iy(\R^m)$-norm. This proves estimate~\eqref{200}
holds after 'restricting' all the operators to the subspace
$E=\Span\{\xi_1,\dots,\xi_m\}$. Since an orthogonal change of the
variables $\xi_1,\dots,\xi_n$ preserves the original
estimate~\eqref{200}, the $m$-dimensional subspace $E$ can be
arbitrary.
\end{proof}

\begin{definition}
A subspace $E\subset\R^n$ is said to be \emph{coordinate} if it
has the form $E=\{x=(x_1,\dots, x_n): x_{i_1}=\dots=x_{i_k}=0\}$,
where $i_1,\dots,i_k\in\{1,\dots,n\}$.
\end{definition}

We denote by $P(\xi)\lceil E$ the restriction of a polynomial
$P(\xi)$ to a coordinate subspace $E$ and by $P(D)\lceil E$ the
corresponding operator.

\begin{cor}\label{suzhenie}
If a system $\Pj$ is weakly coercive in the isotropic space
$\WplRnInf$, then the system $\{P_j(D)\lceil E\}_1^N$ remains
weakly coercive in $\overset{\circ}{W}\rule{0pt}{2mm}_\infty^l(E)$
after restriction to an arbitrary coordinate subspace
$E\subset\R^n$.
\end{cor}

\begin{remark}
We emphasize that the coefficients of the restricted operators
$Q(x,D)$ and $\PjxD$ depend on all $n$ variables as before, while
the differentiation is performed only with respect to the first
$m$ variables. Note also that functions $f\in C_0^\iy(\R^m)$ are
not compactly supported in $\R^n$.
\end{remark}

The following result, announced in~\cite{dan2004}, presents an
analogue of Theorem~\ref{deLeu-Mir_intr} in the case of a
homogeneous system of operators.

\begin{theorem}\label{odnorod}
Suppose $l\ge 2$ and let $\{P_j(D)\}^N_1$  be a system of
operators with constant coefficients of order $l$ satisfying the
following conditions:

\emph{(i)} $n\ge 2N+1$;

\emph{(ii)} the polynomials  $\{P_j^l(\xi)\}_1^N$ restricted to an
arbitrary two-dimensional subspace of  $\R^n$ remain linearly
independent.

Then the system  $\{P_j(D)\}^N_1$  is weakly coercive in the
isotropic Sobolev space $\WplRnInf$ if and only if it is elliptic.
\end{theorem}

\begin{proof} The sufficiency is immediate from Theorem~\ref{th2}.

\noindent \textsl{Necessity.}  Let $\Pj$ be a weakly coercive
system in $\WplRnInf$ that is not elliptic, that is,
$P^l_j(\xi^0)=0$,\; $j\in\{1,\dots,N\}$, for some
$\xi^0=(\xi^0_1,\dots,\xi^0_n)\in\R^n\setminus\{0\}$. Changing the
variables $\xi_1,\dots,\xi_n$ if necessary, we can assume that
$\xi^0=(1,0,\dots,0)$. This means that in each $P_j(\xi)$ the
coefficients of $\xi_1^l$ are zero.

Let $P_j^{l-1}(D):=\sum_{|\al|=l-1} a_{j\al}D^\al$,\
$j\in\{1,\dots,N\}$. We have $P_j^{l-1}(\xi^0)\neq 0$ for some
$j\in\{1,\dots,N\}$ since otherwise, after the substitution
$\xi=\xi^0 t$,\ $t>0$, in the algebraic
inequality~\eqref{alg_ineq}, which follows from the
estimate~\eqref{200} with $Q=D_1^{l-1}$ (see
Proposition~\ref{prop_alg_ineq}), we arrive at a contradiction as
$t\to +\iy$.

After a linear transformation of the system  $\{P_j(\xi)\}_1^N$ we
can assume that
\begin{equation} \label{8a}
P_1^{l-1}(\xi^0)=1,\quad P_2^{l-1}(\xi^0)=0,\quad\dots,\quad
P_N^{l-1}(\xi^0)=0.
\end{equation}
Since the monomial $\xi_1^l$ is missing from every of polynomials
$P_j^l$ and since these polynomials are homogeneous, it follows
that
\begin{equation} \label{m1}
\frac{\partial P_j^l}{\partial\xi_1}(\xi^0)=0,\qquad
j\in\{1,\dots,N\}.
 \end{equation}

Further, because $n\ge 2N+1$, by
Proposition~\ref{prop_weak_coerc}, (iii) the Jacobi matrix of the
map $P^l=\(P_1^l,\dots,P_N^l\):\R^n\to\R^{2N}$ at $\xi^0$ has rank
at most $2N-1$. This means that there exists a vector
$\la:=(0,\la_2,\dots,\la_n)\in\R^n\setminus\{0\}$ such that
$$
  \sum_{k=2}^n \la_k \frac{\partial
  P_j^l}{\partial\xi_k}(\xi^0)=0, \qquad j\in\{1,\dots,N\}.
$$
If necessary making an orthogonal change of the variables
$\xi_1,\dots,\xi_n$ of the form
$$
   \xi'_k=\sum_{r=1}^n c_{kr}\xi_r,\qquad k,\ r\in\{1,\dots,n\},
$$
where  $C:=\(c_{kr}\)_{n\times n}$ is an orthogonal matrix with
the first two rows consisting of the coordinates of the vectors
$e_1=(1,0,\dots,0)$ and $e_2=
\la/|\la|=(0,\la_2/|\la|,\dots,\la_n/|\la|)$, we obtain
$$
 \frac{\partial P_j^l}{\partial e_1}(\xi^0)=
 \frac{\partial P_j^l}{\partial \xi_1}(\xi^0)=0,\qquad
 \frac{\partial P_j^l}{\partial e_2}(\xi^0)=\sum_{k=2}^n
 \frac{\la_k}{|\la_k|} \frac{\partial P_j^l}{\partial \xi_k}(\xi^0)=0.
$$
In addition to~\eqref{m1} we may assume that
\begin{equation} \label{m2}
\frac{\partial P_j^l}{\partial\xi_2}(\xi^0)=0,\qquad
j\in\{1,\dots,N\}.
\end{equation}
Relations~\eqref{m1} and~\eqref{m2} mean that $\{P_j(\xi)\}_1^N$
contains neither monomials $\xi_1^l$ nor $\xi_1^{l-1}\xi_2$.

Consider the 'restriction' of the system  $\Pj$ to the subspace
$E=\Span\{\xi_1,\xi_2\}$. By Corollary~\ref{suzhenie} it remains
weakly coercive in
$\overset{\circ}{W}\rule{0pt}{2mm}_\infty^l(\R^2)$. We keep the
same notation for the 'restricted' objects. Taking account of
relations~\eqref{8a}--\eqref{m2} and applying
Proposition~\ref{Q=MP+N} to the operators  $\Pj$ 'restricted' to
$E$ we obtain
\begin{equation}\label{y}
\xi_1^{l-1}= \sum_{j=1}^N M_j(\xi)\[a_j\xi_1^{l-2}\xi_2^2+ \dots +
\delta_j^1 \xi_1^{l-1}+\ldots \]+M_{N+1}(\xi),
\end{equation}
where the $M_j(\cdot)$,\ $j\in\{1,\dots,N+1\}$, are multipliers on
$L^\iy(\R^2)$ and $\delta_j^1$ is the Kronecker delta.

Dividing both sides of~\eqref{y} by $\xi_1^{l-1}$, we arrive at
\begin{equation}\label{yy}
\lim_{\xi_1\to +\infty} M_1(\xi_1, \xi^0_2)=1,\qquad \xi^0_2
=\const\in\R.
\end{equation}
We claim that~\eqref{yy} contradicts Lemma~\ref{eberlein}. Indeed,
by assumption the leading forms $P_j^l(\xi)$ remain linearly
independent after 'restriction' to $E$; therefore, by
Proposition~\ref{c_j},\ $\mu_j(0)=0$,\ $j\in\{1,\dots,N\}$, where
the $\mu_j$ are the finite measures in the integral
representation~\eqref{2.7A} for the multipliers $M_j=\hat\mu_j$
involved in~\eqref{y}. By Proposition~\ref{eberlein} some convex
combinations of 'shifts' of the function  $M_1(\xi)$ converge
uniformly to the constant function $\mu_1(0)=0$, that is,
\begin{equation} \label{m3}
\begin{aligned}
  &\sum_{k=1}^m c_k M_1\(\xi-\zeta^{(k)}\) \rightrightarrows 0; \\
   \sum_{k=1}^m c_k=1,\qquad & c_k>0, \quad
  \zeta^{(k)}\in\R^2,\quad k\in\{1,\dots,m\}.
\end{aligned}
\end{equation}
It follows from~\eqref{m3} that for $\eps=1/2$ there exist $R>0$
and $a_1,\dots,a_m\in\R$ such that
\begin{equation} \label{m4}
  \left| \sum_{k=1}^m c_k M_1(\xi_1,a_k)
  \right|\le \frac 1 2,\qquad \xi_1>R.
\end{equation}
But inequality~\eqref{m4} contradicts relation~\eqref{yy}. The
proof is complete.
\end{proof}

\begin{remark}
(i) Condition  (ii) of Theorem~\ref{odnorod} is essential. For
instance, condition (i) holds for the system
$P_1(\xi):=\xi_1^2+\xi_2^2 + \xi_3^2$,\ $P_2(\xi):=(\xi_4 +
i)(\xi_5+i)$\ ($n=2N+1= 5$), but condition (ii) fails: the
restrictions of the polynomials $\{P_j^2(\xi)\}_1^2$ to the
two-dimensional subspace $\Span\{\xi_1,\xi_2\}$ are linearly
dependent. The system  $\{P_j(D)\}_1^2$ is weakly coercive in
$\overset{\circ}{W}\rule{0pt}{2mm}_\iy^2(\R^5)$ but not elliptic.

At the same time, condition (ii) in Theorem~\ref{odnorod} is not
necessary.  For instance, the system $P_1(\xi):=\xi_1^2+\xi_2^2$,\
$P_2(\xi):=\xi_3^2+\xi_4^2+\xi_5^2$ is weakly coercive in
$\overset{\circ}{W}\rule{0pt}{2mm}_\iy^2(\R^5)$, although the
restrictions of the polynomials $\{P_j(\xi)\}_1^2$ to the subspace
$\Span\{\xi_1,\xi_2\}$ are linearly dependent.

(ii)  We do not have any examples of systems of operators which
are weakly coercive in
$\overset{\circ}{W}\rule{0pt}{2mm}_\iy^2(\R^n)$, but not elliptic
for $N>1$, for which condition  (i) fails but (ii) holds. However,
it is easy to construct systems failing both conditions (i) and
(ii) of Theorem~\ref{odnorod}. For instance, for $n=2N$ both
conditions (i) and (ii) of Theorem \ref{odnorod} fail for the
system $P_j(\xi):=(\xi_{2j-1}+i)(\xi_{2j}+i)$,\
$j\in\{1,\dots,N\}$. This system is also weakly coercive in
$\overset{\circ}{W}\rule{0pt}{2mm}_\iy^2(\R^n)$, but not elliptic.
\end{remark}

\subsection{A generalization of the de Leeuw-Mirkil theorem to
operators with variable coefficients}

By Theorem~\ref{deLeu-Mir_intr}, which is due to de Leeuw and
Mirkil, if an operator $P(D)$ is weakly coercive in $\WplRnInf$
for $n\ge 3$, then it is elliptic. The next theorem extends this
result to operators with variable coefficients.

\begin{theorem} \label{de_Leeuw_vary}
Suppose that $l\ge 2$,\ $n\ge 3$, and let $P(x,D)$ be a
differential operator of order $l\ge 2$ in which $a_\al(\cdot)\in
L^\iy(\R^n)$ for $|\al|\le l-1$,\ $a_\al(\cdot)\in C^1(\R^n)$ for
$|\al|=l-1$, and $a_\al=\const$ for $|\al|=l$, that is,
$P^l(x,D)=P^l(D)$.

Then the operator P(x,D) is weakly coercive in the isotropic
Sobolev space $\WplRnInf$ if and only if it is elliptic.
\end{theorem}

\begin{proof} The necessity follows from Theorem~\ref{th2}.

\noindent \textsl{Sufficiency.} Suppose that the operator $P(x,D)$
is weakly coercive in $\WplRnInf$, that is,
\begin{equation} \label{1a}
\sum_{|\al|<l} \|D^\al f\|_{L^\iy(\R^n)} \le C_1
\|P(x,D)f\|_{L^\iy(\R^n)} +C_2 \|f\|_{L^\iy(\R^n)},\qquad f\in
C_0^\iy(\R^n).
\end{equation}
If $P$ is not elliptic, then $P^l(\xi^0)=0$ for
$\xi^0=(\xi^0_1,\dots,\xi^0_n)\in\R^n\setminus\{0\}$. Changing the
variables $\xi_1,\dots,\xi_n$ if necessary, we can assume that
$\xi^0=(1,0,\dots,0)$.

By Euler's identity $\sum_{k=1}^n \xi_k (\partial P^l/\partial
\xi_k)=nP^l$ for the homogeneous polynomial $P^l(\xi)$, the
condition $P^l(\xi^0)=0$ implies the relation $(\partial
P^l/\partial \xi_1)(\xi^0)=0$. However, since $n\ge 3$, it follows
from Proposition~\ref{prop_weak_coerc}, (iii) that the Jacobi
matrix of the map $P^l=(\Re P^l,\ \Im P^l): \R^n\to\R^2$ at the
point $\xi^0$ has rank at most $1$. Making a suitable linear
change of the coordinates $\xi_2,\dots,\xi_n$ if necessary (see
the proof of Theorem~\ref{odnorod}) we can assume that the second
column of the Jacobi matrix is zero, that is, $(\partial
P^l/\partial\xi_2)(\xi^0)=0$. Thus, the symbol $P(x,\xi)$ of the
operator $P(x,D)$ does not contain the monomials $\xi_1^l$ or
$\xi_1^{l-1}\xi_2$.

Further, combining Proposition~\ref{suzh} with estimate~\eqref{1a}
yields the 'restricted' estimate
\begin{equation} \label{1b}
\sum_{\al_1+\al_2<l} \|D_1^{\al_1}D_2^{\al_2}f\|_{L^\iy(\R^n)} \le
C_1 \|P(x,D_1,D_2,0,\dots,0)f\|_{L^\iy(\R^n)} + C_2
\|f\|_{L^\iy(\R^n)}
\end{equation}
for all $f\in C_0^\iy(\R^2)$. Note that
$P^l(\xi_1,\xi_2,0,\dots,0)\not\equiv 0$, for otherwise
estimate~\eqref{1b} (see Proposition~\ref{mal_neodn} and
Remark~\ref{rem_mal_neodn}, (i)) implies the relations
$$
\xi^\al=\la_\al(x) P^{l-1}(x,\xi_1,\xi_2,0,\dots,0)
$$
for all $|\al|=l-1$, which is obviously impossible. Hence there
exists $k\ (\ge 2)$ such that the coefficient of
$\xi_1^{l-k}\xi_2^k$ in the polynomial
$P^l(\xi_1,\xi_2,0,\dots,0)$ differs from zero.

We take the minimum such $k$ and draw a line $\omega$ through the
points $(l-1,0)$ and $(l-k,k)$. It is not vertical since $k\ge 2$.
We denote by $l':=(l_1',l_2')$ the vector with components
$l_1':=l-1$ and $l_2':=k(l-1)/(k-1)$ equal to the lengths of the
intercepts of $\omega$ with the coordinate axes. The 'restricted'
operator has the following form:
$$
  P(x,D_1,D_2,0,\dots,0)=\sum_{\al_1/l_1'+\al_2/l_2'\le 1}
  a_{(\al_1,\al_2)}(x)D_1^{\al_1} D_2^{\al_2},
$$
that is, $\omega$ is an $l'$-principal line for the operator $P$.
Indeed, there are no terms $\xi_1^l$ or $\xi_1^{l-1}\xi_2$ in the
symbol $P(x,\xi_1,\xi_2,0,\dots,0)$ and there are no points with
integer coordinates in the strip $l-1\le x+y\le l,\ x,y\ge 0$,
except on the lines $x+y=l-1$ and $x+y=l$, and the line interval
with end-points $(l-1,0)$ and $(l-k,k)$ lies entirely in this
strip (see Fig.~1).

\begin{center}
\begin{picture}(100,100)(-10,-10)
\thicklines {\small\put(0,0){\vector(1,0){100}}
 \put(0,0){\vector(0,1){90}}
 \put(-7,5){$0$}  \put(90,5){$x$} \put(5,82){$y$}
 \put(70,0){\line(-1,1){70}}
 \put(70,-10){$l$} \put(-7,68){$l$}
 \put(70,0){\circle{3}}
 \put(60,10){\circle{3}}
 \put(50,20){\circle{3}}
\put(60,0){\circle*{3}}  \put(40,30){\circle*{3}}
\put(45,30){$(l\!-\!k,k)$}
 \put(20,60){\line(2,-3){45}}
 \put(25,60){$\omega$}
 \thinlines
 \put(60,0){\line(-1,1){60}}
 }
 \put(-60,-5){Fig.~1}
\end{picture}
\end{center}

It follows that the $l'$-principal part of the operator
$P(x,D_1,D_2,0,\dots,0)$ has the form
$$
P^{l'}(x,D_1,D_2,0,\dots,0)=c(x) D_1^{l-1}+ b D_1^{l-k} D_2^k,
$$
where $c(x):=a_{l-1,0,0,\dots,0}(x)$ and $b:=a_{l-k,k,0,\dots,0}$,
and where we have $b\neq 0$.

Since the estimate~\eqref{1b} holds with the operator $D_1^{l-1}$
on the left-hand side, it follows by Proposition~\ref{mal_neodn}
and Remark~\ref{rem_mal_neodn}, (i) that
\begin{equation} \label{1c}
\xi_1^{l-1}=\la(x)\[c(x)\xi_1^{l-1}+b\xi_1^{l-k}\xi_2^k\],
 \qquad x\in\R^n,\quad \xi_1,\xi_2\in\R.
\end{equation}
Relation~\eqref{1c} implies that $\la(x)c(x)\equiv 1$ and
$\la(x)b\equiv 0$, which contradicts the condition $b\neq 0$. This
contradiction proves that the operator $P(x,D)$ is elliptic.
\end{proof}

\begin{remark}
In the space $L^p(\Om)$,\ $p\in[1,\iy]$, each differential
expression $P(x,D)$ of the form~\eqref{8} is naturally associated
with a minimal and a maximal differential operators $P_{\min}$ and
$P_{\max}$. Recall (see~\cite{Ber}, Ch.\,2, \pr\,2) that by
definition $P_{\min}$ is the closure in $L^p(\Om)$ of the
differential operator $P'=P\lceil C_0^\iy(\Om)$  defined
originally on the domain $\dom(P')=C_0^\iy(\Om)$. Clearly,
$\WplOm\subset\dom(P_{\min})$.
\end{remark}

In addition, the coercivity criterion in $\WplOmInf$ for
$p\in(1,\iy)$ (Theorem~\ref{Besov_th}) implies that for an
operator $P(x,D)$ with continuous coefficients in a bounded domain
$\ol\Om$ the relation $\dom(P_{\min})=\WplOmInf$ is equivalent to
the ellipticity of $P(x,D)$.

Thus, Theorem~\ref{de_Leeuw_vary} is equivalent to the following
result.

\begin{cor}
Under the assumptions of Theorem~\ref{de_Leeuw_vary} the inclusion
$$
\dom(P_{\min})\subset\overset{\circ}{W}
\rule{0pt}{2mm}_\iy^{l-1}(\R^n)
$$
is equivalent to the ellipticity of the operator $P(x,D)$.
\end{cor}

\begin{proof}
Since the operator $P_{\min}$ is closed, the inclusion
$\dom(P_{\min})\subset\overset{\circ}{W}\rule{0pt}{2mm}_\iy^{l-1}(\R^n)$
is equivalent to estimate~\eqref{1a}, that is, to the weak
coercivity of $P_{\min}$. It remains to apply
Theorem~\ref{de_Leeuw_vary}.
\end{proof}

\begin{remark}
(i) We emphasize that it is because the selection of a principal
part of a differential operator is not unique that we can use the
anisotropic version of Proposition~\ref{mal_neodn} for the proof
of 'isotropic' Theorem~\ref{de_Leeuw_vary}.

(ii)  In the case when the operator $P$ has constant coefficients
the conditions
$$
P^l(1,0,\dots,0)=\frac{\partial P^l}{\partial
\xi_1}(1,0,\dots,0)=\frac{\partial P^l}{\partial
\xi_2}(1,0,\dots,0)=0
$$
after 'restricting' $P$ to the two-dimensional subspace
$\Span\{\xi_1,\xi_2\}$ mean that $\nabla \td P^l(1,0)=0$ ($\td P$
is the corresponding 'restriction' of the operator $P$). The last
condition immediately contradicts
Proposition~\ref{prop_weak_coerc}, (iv), and the final part of the
proof of Theorem~\ref{de_Leeuw_vary} can be omitted.
\end{remark}

\section{A characterization of weakly coercive \\ operators of two
variables in $\protect\WplRTwo$,\ $p\in[1,\iy]$}

In~\cite{LeeMir}, p. 123 the authors give Malgrange's example of
an operator that is weakly coercive in $\WpTwoRTwoInf$, but not
elliptic: $P(D)=(D_1+i)(D_2+i)$.

The following assertion, in particular, gives a complete
characterization of weakly coercive operators in the isotropic
Sobolev space $\WplRTwoInf$.

\begin{theorem}\label{obschij_vid}
\emph{(i)} An arbitrary weakly coercive operator $P(D)$ of order
$l\ge 2$ in the isotropic space $\WplRTwoInf$ has the form
\begin{equation}\label{dve_perem}
P(D)=R(D)\prod_{k=1}^m \(\la_k D_1+\mu_k D_2+\al_k\)+Q(D),
\end{equation}
where $R(D)$ is an elliptic operator of order $l-m$,\ $Q(D)$ is an
operator of order $\le l-2$,\ $\al_k\in\C\setminus\R$,\
$(\la_k,\mu_k)\in\R^2$ are pairwise non-collinear vectors, where
$k\in\{1,\dots,m\}$, and $m\le l$.

\emph{(ii)} Conversely, any operator of the form~\eqref{dve_perem}
is weakly coercive in $\WplRTwo$,\ $p\in[1,\iy]$.
\end{theorem}

\begin{proof}
(i) We assume first that $P(\xi)$ is an arbitrary polynomial of
order $l$ and $P^{l-1}(\xi):=\sum_{|\al|=l-1} a_\al \xi^\al$ is
the $(l-1)$-homogeneous part of the polynomial $P(\xi)$. The
principal form $P^l(\xi)$ can be represented as follows:
$$
P^l(\xi_1,\xi_2)=\prod_{j=1}^s (a_j\xi_1+b_j\xi_2)^{k_j},
$$
where $k_j\ge 1$,\; $\sum_{j=1}^s k_j=l$, and $(a_j,b_j)\in\C^2$
are pairwise non-collinear, where $j\in\{1,\dots,s\}$. We claim
that $P(\xi)$ can be expressed as
\begin{equation}\label{predstavl_ot_2_perem}
P(\xi_1,\xi_2)=\prod_{j=1}^s
\[(a_j\xi_1+b_j\xi_2)^{k_j}+Q_j(\xi_1,\xi_2)\]+Q(\xi_1,\xi_2),
\end{equation}
where $\deg Q_j< k_j$,\; $j\in\{1,\dots,s\}$,\; $\deg Q\le l-2$.
In fact, the rational fraction $P^{l-1}(\xi)/P^l(\xi)$, which is
the ratio of two homogeneous polynomials of two variables, can be
decomposed to a sum of partial fractions:
$$
 \frac{P^{l-1}(\xi_1,\xi_2)}{P^l(\xi_1,\xi_2)}=
 \sum_{j=1}^s \frac{Q_j(\xi_1,\xi_2)}{(a_j\xi_1+b_j\xi_2)^{k_j}},
 \qquad\text{where}\quad \deg Q_j<k_j.
$$
Clearly, this implies that the homogeneous forms of orders $l$ and
$l-1$ of the polynomial
 $$
\tilde P(\xi_1,\xi_2):=P^l(\xi_1,\xi_2) \prod_{j=1}^s
\[1+\frac{Q_j(\xi_1,\xi_2)}{(a_j\xi_1+b_j\xi_2)^{k_j}}\]=
\prod_{j=1}^s
\[(a_j\xi_1+b_j\xi_2)^{k_j}+Q_j(\xi_1,\xi_2)\]
$$
coincide with $P^l(\xi)$ and $P^{l-1}(\xi)$, respectively.
Therefore, the difference
$$
Q(\xi):=P(\xi)-\tilde P(\xi)
$$
is a polynomial of degree $\le l-2$, which proves that the
representation~\eqref{predstavl_ot_2_perem} holds.

Now let $P(D)$ be a weakly coercive operator in $\WplRTwoInf$. By
Proposition~\ref{prop_weak_coerc}, (iv) the polynomial $P^l(\xi)$
has no multiple real zeros, and hence
$$
P^l(\xi_1,\xi_2)=\prod_{j=1}^{s-m} (a_j\xi_1+b_j\xi_2)^{k_j}
\prod_{j=1}^m \(\la_j \xi_1+\mu_j \xi_2\), \qquad k_j\ge 1,\quad
\sum_{j=1}^{s-m} k_j=l-m.
$$
Here the vectors $(\la_j,\mu_j)\in\R^2$,\ $j\in\{1,\dots,m\}$, and
$(a_j,b_j)\in\C^2$,\ $j\in\{1,\dots,s-m\}$, are pairwise
non-collinear. Now we write the
decomposition~\eqref{predstavl_ot_2_perem} for the polynomial
$P(\xi)$, with $a_{s-m+j}=\la_j$,\ $b_{s-m+j}=\mu_j$ and
$Q_{s-m+j}(\xi)\equiv \al_j$,\ $j\in\{1,\dots,m\}$, since
$k_{s-m+1}=\dots=k_s=1$. To complete the proof it suffices to note
that $\al_j\not\in\R$,\; $j\in\{1,\dots,m\}$, by
Proposition~\ref{prop_weak_coerc}, (i) and to set
$$
R(D):= \prod_{j=1}^{s-m} \[(a_j D_1+b_j D_2)^{k_j}+Q_j(D)\].
$$
The operator $R(D)$ has order $\sum_{j=1}^{s-m} k_j=l-m$ and is
elliptic because its principal part $R^{l-m}(\xi)=
\prod_{j=1}^{s-m} (a_j\xi_1+b_j\xi_2)^{k_j}\neq 0$ for
$\xi\in\R^2\setminus\{0\}$. Taking~\eqref{predstavl_ot_2_perem}
into account we see that $P(D)$ has the form~\eqref{dve_perem}.

(ii) Now we prove that the operator~\eqref{dve_perem} is weakly
coercive in $\WplRTwo$ for any $p\in[1,\iy]$. By
Proposition~\ref{prop_weak_coerc}, (ii) we may assume that
$Q(D)=0$. First, using induction on $m$,\; $m\in\{0,\dots,l\}$,
we prove that
\begin{equation}\label{mult_phi}
\Phi_\ga^{(m)}(\xi)=\Phi_{\ga_1,\ga_2}^{(m)}(\xi):=\chi(\xi)
\frac{\xi_1^{\ga_1}\xi_2^{\ga_2}}{R(\xi)\prod_{k=1}^m \(\la_k
\xi_1+\mu_k \xi_2+\al_k\)}\in\cm_1.
\end{equation}
Here $|\ga|=\ga_1+\ga_2<l'+m$,\ $R(\xi)$ is an elliptic polynomial
of degree $l'$, and $\chi(\xi)$ is the corresponding 'cutoff'
function.

After an orthogonal change of the variables $\xi_1,\dots,\xi_n$ we
may assume that $\mu_m=0$. Since this change preserves the
non-collinearity of the vectors $\(\la_k,\mu_k\)$,\;
$k\in\{1,\dots,m\}$, we conclude that $\mu_k\neq 0$ for all
$k\in\{1,\dots,m-1\}$.

For $m=0$ the assertion in question is obvious:
$$
 \Phi_\ga^{(0)}(\xi)=\chi(\xi)\xi^\ga(R(\xi))^{-1}\in\cm_1
 \quad\text{for}\quad |\ga|<l',
$$
because the polynomial $R(\xi)$ is elliptic (see~\cite{BDLM},
\pr\,4).

Further, we have $(\la_k\xi_1+\mu_k\xi_2+\al_k)^{-1}\in\cm_1$,\;
$k\in\{1,\dots,m\}$. Indeed, let $\chi_+(\cdot)$ be the Heaviside
function and let $\de(\cdot)$ be the Dirac measure on the line.
The Fourier-Stieltjes transform $\hat\sigma$ of a finite measure
$\sigma$ with density $-i\chi_+(t_1)e^{-t_1}\otimes\de(t_2)$ is
\begin{equation}\label{xi1+i}
\hat\sigma=-i\int_{\R^2}
\chi_+(t_1)\de(t_2)e^{-t_1}e^{it_1\xi_1}e^{it_2\xi_2}\, dt_1
dt_2=-i\int_0^{+\iy} e^{it_1(\xi_1+i)}\, dt_1=(\xi_1+i)^{-1}.
\end{equation}
Making the change of the variables
$$
\eta_1:=\la_k\xi_1+\mu_k\xi_2+\al_k,\qquad
\eta_2:=-\mu_k\xi_1+\la_k\xi_2
$$
in~\eqref{xi1+i}, we obtain
$(\la_k\xi_1+\mu_k\xi_2+\al_k)^{-1}=\hat\sigma_1$, where the
finite measure $\sigma_1$ has density
$$
 \frac{\Im\ \al_k}{\la_k^2+\mu_k^2}\chi_+
 \(\Im\ \al_k(t_1\la_k+t_2\mu_k)\)
 \exp\[(i\Re\ \al_k-\Im\
 \al_k)\frac{t_1\la_k+t_2\mu_k}{\la_k^2+\mu_k^2}\]
 \otimes\de(-t_1\mu_k+t_2\la_k)
$$
at the point $(t_1,t_2)\in\R^2$. This yields the inclusion
$(\la_k\xi_1+\mu_k\xi_2+\al_k)^{-1}\in\cm_1$ (see~\cite{St},
Ch.\,IV, \pr\,3).

Suppose that~\eqref{mult_phi} holds for all the functions
$\Phi_\ga^{(t)}(\xi)$ with $t\le m-1$. If $\ga_1>0$, then
$\Phi_\ga^{(m)}\in\cm_1$ as it is the product of
$\Phi_{\ga_1-1,\ga_2}^{(m-1)}$ and $\xi_1(\la_m\xi_1+\al_m)^{-1}$.
If $\ga_1=0$ and $\ga_2<l'+m-1$, then $\Phi_\ga^{(m)}\in\cm_1$ as
it is the product of $\Phi_{0,\ga_2}^{(0)}$ and $\prod_{k=1}^m
(\la_k \xi_1+\mu_k \xi_2+\al_k)^{-1}$.

Now let $\ga_1=0$,\ $\ga_2=l'+m-1$. Let
$R(\xi)=a_0\xi_2^{l'}+\dots$,\ $a_0\neq 0$, where the dots stand
for a polynomial of degree less than $l'$ with respect to $\xi_2$.
Consider the difference between the function
$\Phi_{0,l'+m-1}^{(m)}(\xi)$ and the multiplier $\chi(\xi)
\[a_0\mu_1\dots\mu_{m-1}(\la_m\xi_1+\al_m)\]^{-1}$:
\begin{equation} \label{3''}
 \chi(\xi)\frac{a_0\mu_1\dots\mu_{m-1}\xi_2^{l'+m-1}-\(a_0\xi_2^{l'}+\dots\)
 \prod_{k=1}^{m-1} \(\la_k \xi_1+\mu_k\xi_2+\al_k\)}
{a_0\mu_1\dots\mu_{m-1}R(\xi)\prod_{k=1}^m \(\la_k \xi_1+\mu_k
   \xi_2+\al_k\)}\ .
\end{equation}
The polynomial in the numerator of~\eqref{3''} has degree less
than $l'+m-1$ with respect to $\xi_2$. Therefore, this fraction is
a linear combination of functions $\Phi_\ga^{(m)}(\xi)$ for which
$\ga_2<l'+m-1$, so that it is a sum of multipliers by the above.

Now the weak coercivity of the operator $P(D)$ in $\WplRTwo$
follows from the identities
$$
  \xi^\ga \hat f(\xi)= \Phi^{(m)}_\ga(\xi) P(\xi)\hat f(\xi)+
  \xi^\ga (1-\chi(\xi)) \hat f(\xi),\qquad f\in
  C_0^\iy(\R^2),\qquad |\ga|<l,
$$
to which we apply the inverse Fourier transform while
taking~\eqref{mult_phi} into account.
\end{proof}

\begin{cor}
The space $L^0_{\iy,\R^2}(P)$, where $P(D)$ is the operator
in~\eqref{dve_perem}, consists of operators of the following form:
\begin{equation}\label{L0P}
 T(D):=cR^{l-m}(D)\prod_{k=1}^m \(\la_k D_1+\mu_k D_2\)+Q(D),\qquad
c\in\C,
\end{equation}
where $Q(D)$ is an arbitrary operator of order $\le l-1$ and
$R^{l-m}(D)$ is the principal part of the operator $R(D)$ (of
order $l-m$).
\end{cor}

\begin{proof}
By Proposition~\ref{mal_neodn}, if $T\in L^0_{\iy,\R^2}(P)$, then
$$
 T^l(D)=cP^l(D)=cR^{l-m}(D)\prod_{k=1}^m \(\la_k D_1+\mu_k D_2\);
$$
hence the operator $T(D)$ has the form~\eqref{L0P}.

Conversely, let $T(D)$ be the operator of the form~\eqref{L0P}. By
Theorem~\ref{obschij_vid}, (ii), $T(D)$ is weakly coercive in
$\WplRTwoInf$, and for $c=0$ we have $Q=T\in L^0_{\iy,\R^2}(P)$.
If $c\neq 0$, then the corresponding 'cutoff' function satisfies
$$
\chi(\xi)\frac{T(\xi)}{P(\xi)}=
\chi(\xi)\frac{cP(\xi)+(T(\xi)-cP(\xi))}{P(\xi)}=
\chi(\xi)\[c+\frac{T(\xi)-cP(\xi)}{P(\xi)}\]\in\cm_1,
$$
because $T^l=cP^l$. Hence $\deg(T-cP)\le l-1$, so $T\in
L^0_{\iy,\R^2}(P)$.
\end{proof}

\begin{cor} \label{cor_10}
The product of an elliptic operator of order $l$ of two variables
and a weakly coercive operator of order $m$ in
$\overset{\circ}{W}\rule{0pt}{2mm}_\iy^m(\R^2)$ is weakly coercive
in $\overset{\circ}{W}\rule{0pt}{2mm}_\iy^{l+m}(\R^2)$.
\end{cor}

\begin{proof}
Let $T_1(D)$ be an elliptic operator of order $l$ and $T_2(D)$ a
weakly coercive operator of order $m$. By
Theorem~\ref{obschij_vid}, (i),
\begin{equation}\label{T1T2}
T_2(D)=R(D)\prod_{k=1}^s \(\la_k D_1+\mu_k D_2+\al_k\)+Q(D),
\end{equation}
where $R(D)$ is an elliptic operator of order $m-s$,\ $\deg Q\le
m-2$,\ $\al_k\in\C\setminus\R$, the $\(\la_k,\mu_k\)$ are pairwise
non-collinear vectors in $\R^2$,\ $k\in\{1,\dots,s\}$,\ $s\le m$.
Multiplying the~\eqref{T1T2} by $T_1(D)$ we obtain the
representation
\begin{equation*}
T_1(D)T_2(D)=T_1(D)R(D)\prod_{k=1}^s \(\la_k D_1+\mu_k
D_2+\al_k\)+T_1(D)Q(D),
\end{equation*}
which also has the form~\eqref{dve_perem}. In fact, $T_1(D)R(D)$
is an elliptic operator of order $l+m-s$ and $\deg(T_1Q)=\deg
T_1+\deg Q\le l+m-2$.
\end{proof}

The following theorem provides an algebraic criterion for weak
coercivity in $\WplRTwoInf$.

\begin{theorem} \label{th_16}
Let $P(D)$ with $D=(D_1,D_2)$ be an operator of order $l$, and
assume that all the coefficients and the zeros of $P^l(\xi)$ are
real.

\emph{(i)} If $P(D)$ is weakly coercive in the isotropic space
$\WplRTwoInf$, then the polynomials $P^l(\xi)$ and $\Im\
P^{l-1}(\xi)$ have no common non-trivial real zeros.

\emph{(ii)} Conversely, if polynomials $P^l(\xi)$ and $\Im\
P^{l-1}(\xi)$ have no common non-trivial real zeros, then the
operator $P(D)$ is weakly coercive in $\WplRTwo$,\ $p\in[1,\iy]$.
\end{theorem}

\begin{proof} By assumption the principal part $P^l(\xi)$ has the form
\begin{equation} \label{P^l}
P^l(\xi)=\prod_{k=1}^l
\(\la_k\xi_1+\mu_k\xi_2\),\qquad\text{where}\qquad\(\la_k,\mu_k\)\in\R^2,\quad
k\in\{1,\dots,l\}.
\end{equation}

(i) Suppose that the operator $P(D)$ is weakly coercive in
$\WplRTwoInf$. Then the vectors $\(\la_k,\mu_k\)$ in~\eqref{P^l}
are pairwise non-collinear by Proposition~\ref{prop_weak_coerc},
(iv). Combining Theorem~\ref{obschij_vid}, (i) and
relation~\eqref{P^l} shows that $P(\xi)$ has the form
\begin{equation} \label{P}
 P(\xi)=\prod_{k=1}^l \(\la_k \xi_1+\mu_k
\xi_2+\al_k\)+Q(\xi),
\end{equation}
where $\al_k\in\C\setminus\R$,\ $k\in\{1,\dots,l\}$, and $\deg
Q\le l-2$. It follows from~\eqref{P} that
\begin{equation} \label{Pl}
 P^{l-1}(\xi)=\sum_{j=1}^l \al_j\prod_{k\neq
  j}\(\la_k\xi_1+\mu_k\xi_2\).
\end{equation}
Now substituting in~\eqref{Pl} one of the zeros $(-\mu_r,\la_r)$
of the principal part $P^l(\xi)$ we obtain
\begin{equation} \label{Pl-1theta_k}
 P^{l-1}(-\mu_r,\la_r)=\al_r\prod_{k\neq r}
 \(\la_r\mu_k-\la_k\mu_r\).
\end{equation}
Under the assumptions on the numbers $\al_k$ and the vectors
$(\la_k,\mu_k)$ this yields
$$
P^{l-1}(-\mu_r,\la_r)\in\C\setminus\R.
$$
It follows that $P^l(-\mu_r,\la_r)=0$ and $\Im\
P^{l-1}(-\mu_r,\la_r)\neq 0$ for all $r\in\{1,\dots,l\}$.

(ii) Conversely, assume that the polynomials $P^l(\xi)$ and $\Im\
P^{l-1}(\xi)$ have no common non-trivial real zeros, that is, that
$\Im\ P^{l-1}(-\mu_r,\la_r)\neq 0$ for all $r\in\{1,\dots,l\}$. It
follows from the proof of Theorem~\ref{obschij_vid}, (i) that a
polynomial $P(\xi)$ with principal part~\eqref{P^l} can be
represented in the form~\eqref{P}, where $\deg Q\le l-2$, and the
$\al_k\in\C$ are some numbers. In this case the polynomial
$P^{l-1}(\xi)$ is represented by the same formula~\eqref{Pl}. In
view of the relations $P^{l-1}(-\mu_r,\la_r)\in\C\setminus\R$ and
$\la_r\mu_k-\la_k\mu_r \in\R\setminus\{0\}$,\
$k,r\in\{1,\dots,l\}$,\ $k\neq r$, equality~\eqref{Pl-1theta_k}
implies that $\al_k\in\C\setminus\R$,\; $k\in\{1,\dots,l\}$. Now
the weak coercivity of the operator $P(D)$ in $\WplRTwo$ follows
from Theorem~\ref{obschij_vid}, (ii).
\end{proof}

\begin{remark} \label{rem_apr_est}
(i) In Theorem~\ref{th_16} the condition that the coefficients
$a_\al$ of the polynomial $P^l$ must be real is not restrictive:
if $P^l$ has real zeros, then its coefficients have the form
$a_\al=ca_\al'$, where $c\in\C$ and $a_\al'\in\R$.

(ii) With the use of the resultant $R[f,g]$ of polynomials $f$ and
$g$ the conditions of Theorem~\ref{th_16} can be written as
$R\[P^l,\Im\ P^{l-1}\](\xi)\neq 0$,\ $\xi\in\R^n$. In this form
weak coercivity can be verified without knowing the zeros of the
polynomial $P^l(\xi)$.

(iii) Theorem~\ref{obschij_vid} in combination with
Proposition~\ref{prop_weak_coerc} shows that any strictly
hyperbolic operator of order $l$ in two variables becomes weakly
coercive in $\WplRTwoInf$ after a perturbation by a suitable
operator of order $l-1$. In general, perturbations of order $l-2$
cannot produce this result.

(iv) Operators~\eqref{dve_perem} remain weakly coercive in
$\WplOm$ for any domain $\Om\subset\R^2$ (including bounded
domains), but they do not exhaust the entire set of weakly
coercive operators in $\WplOm$.

(v) Theorem~\ref{obschij_vid} supplements the results
of~\cite{Lit}. More specifically, by~\cite{Lit}, p.~220 the
d'Alembertian $\square:=D_1^2-a^2D_2^2$ is not weakly coercive in
$\overset{\circ}{W}\rule{0pt}{2mm}^2_\iy(\Om)$,\ where $\Om$ is a
bounded domain in $\R^2$. However, a suitable perturbation of
$\square$ by lower-order terms makes it weakly coercive in
$\WpTwoRTwoInf$ and hence in
$\overset{\circ}{W}\rule{0pt}{2mm}^2_\iy(\Om)$.
\end{remark}

\section{Weakly coercive non-elliptic homogeneous systems}

Here we show that the product of an arbitrary elliptic system and
a special weakly coercive system is weakly coercive, but not
elliptic. This is not the case for an arbitrary weakly coercive
system (see Remark~\ref{last_remark}).

We denote by $\wt{\cm}_1=\wt{\cm}_1(\R^n)$ the class of
multipliers satisfying the conditions of Theorem~\ref{th1}.
Following~\cite{Bom}, \pr\,2 we also introduce a partial ordering
in the set of multi-indices $\Z_+^n$: we will write $\al\le\be$ if
$\al_j\le\be_j$ for all $j\in\{1,\dots,n\}$; moreover, $\al<\be$,
if $\al_j<\be_j$ at least for one $j$.

In some cases the following result makes it easier to verify the
assumptions of Theorem~\ref{th1}

\begin{prop} \label{bez_proizv}
Let $\al\in\Z_+^n\setminus\{0\}$ and let $P(\xi)$ be a polynomial
of degree $l$. Suppose that the zero set of $P(\xi)$ lies in a
ball $B_r^n$. Consider the family of functions
\begin{equation}\label{rac_drob}
\Phi_\be(\xi):=\chi(\xi)\frac{\xi^\be}{P(\xi)},\qquad 0<\be \le
\al,   \quad |\al|< l,
    \end{equation}
where $\chi(\xi)\in C_0^\iy(\R^n)$,\ $0\le\chi(\xi)\le 1$, is a
'cutoff' function equal to zero in $B_r^n$ and to one for
$|\xi|\ge r_1>r$. For $|\xi|\ge r_1$ suppose that the following
conditions hold:

\emph{(i)} the functions~\eqref{rac_drob} satisfy
inequality~\eqref{dv1};

\emph{(ii)} the polynomial $P(\xi)$ satisfies the relations
\begin{equation} \label{ineq_new}
 \prod_{j=1}^n \(1+|\xi_j|\)^{\ga_j}
 \left|(D^\ga P)(\xi)\right|
 \le C|P(\xi)|,\qquad \ga=(\ga_1,\dots,\ga_n)\in\Z_2^n,\qquad C>0.
\end{equation}

Then $\Phi_\be\in\wt{\cm_1}$ whenever $\be\in\Z_+^n$,\;
$0<\be\le\al$.
\end{prop}

\begin{proof} Note first that relations~\eqref{dv1}
and~\eqref{dv2} are met for $|\xi|\le R$, where $R>0$ is
arbitrary, by the continuity of the functions $\Phi_\be$ and their
derivatives. Therefore we shall assume that $|\xi|$ is
sufficiently large (so we do not require the 'cutoff' function
$\chi(\xi)\equiv 1$ for $|\xi|\ge r_1$ in what follows). We also
assume that $0<\be\le\al$ and denote by $C$ various positive
constants.

Consider the case $n=2$ (for $n\ge 2$ the proof is similar).

(i) Assume that $|\xi_1|>1,\ |\xi_2|>1$. Then
relations~\eqref{dv1},~\eqref{dv2} and~\eqref{ineq_new} are
equivalent, respectively, to the following groups of relations:
\begin{gather}
 \label{1'} |\xi_1\xi_2|^\delta|\Phi_\be(\xi)|\le C;\\
 \label{a'} |\xi_1|^{\delta+1}|\xi_2|^\delta|D_1\Phi_\be|\le C,\qquad
  |\xi_1|^\delta|\xi_2|^{\delta+1}|D_2\Phi_\be|\le C,\qquad
 |\xi_1\xi_2|^{\delta+1}|D_1D_2\Phi_\be|\le C;\\
   \label{2a'} |\xi_1\cdot D_1 P|\le C|P(\xi)|,\qquad
  |\xi_2\cdot D_2 P|\le C|P(\xi)|,\qquad
  |\xi_1\xi_2\cdot D_1 D_2 P|\le C|P(\xi)|.
\end{gather}
In view of Theorem~\ref{th1} it suffices to show that~\eqref{1'}
and~\eqref{2a'} imply~\eqref{a'}.

We find an estimate for $D_1\Phi_\be$. We have
$$
  |(D_1\Phi_\be)(\xi)|=
  \left|\frac{\be_1 \xi^\be}{\xi_1 P(\xi)} - \frac{\xi^\be}{P(\xi)}
  \cdot \frac{D_1 P}{P} \right|=
 \left|\frac{\be_1\Phi_\be}{\xi_1}-\Phi_\be\cdot \frac{D_1 P}{P}\right|
   \le
  C \left|\frac{\Phi_\be(\xi)}{\xi_1}\right| \le
  \frac{C}{|\xi_1|^{\delta+1} |\xi_2|^\delta}\ .
$$
We obtain an estimate for $D_2\Phi_\be$ if we interchange $\xi_1$
and $\xi_2$ .

In a similar way we estimate $D_1 D_2 \Phi_\be$. We have
$$
|D_1 D_2 \Phi_\be|=
\left|\frac{\be_1D_2\Phi_\be}{\xi_1}-D_2\Phi_\be\cdot\frac{D_1
P}{P}- \Phi_\be \(\frac{D_1D_2P}{P} -
  \frac{D_1 P}{P}\cdot\frac{D_2 P}{P}\) \right|
 \le \frac{C}{|\xi_1 \xi_2|^{\delta+1}}.
 $$

(ii) Assume that $|\xi_1|\le 1$ and $|\xi_2|>1$. Then
relations~\eqref{dv1},~\eqref{dv2} and~\eqref{ineq_new} are
equivalent, respectively, to the following groups of relations:
\begin{gather}
  \label{1''} |\xi_2|^\delta|\Phi_\be(\xi)|\le C;\\
  \label{a''} |\xi_2|^\delta|D_1\Phi_\be|\le C,\qquad
  |\xi_2|^{\delta+1}|D_2\Phi_\be|\le C,\qquad
 |\xi_2|^{\delta+1}|D_1D_2\Phi_\be|\le C;\\
   \label{2a''} |D_1 P|\le C|P(\xi)|,\qquad
  |\xi_2\cdot D_2 P|\le C|P(\xi)|,\qquad
|\xi_2\cdot D_1 D_2 P|\le C|P(\xi)|,
\end{gather}
We shall show that~\eqref{1''} and~\eqref{2a''} imply~\eqref{a''}.
We set $\be':=(\be_1-1,\be_2)$ for $\be_1>0$ and $\be':=\be$ for
$\be_1=0$. We find estimates for $D_1\Phi_\be$,\ $D_2\Phi_\be$ and
$D_1 D_2\Phi_\be$:
\begin{gather*}
  |(D_1\Phi_\be)(\xi)|=
  \left|\be_1 \Phi_{\be'}(\xi) - \Phi_\be(\xi)
  \cdot \frac{D_1 P}{P} \right| \le
  C \(|\Phi_{\be'}(\xi)|+|\Phi_\be(\xi)|\) \le
  \frac{C}{|\xi_2|^\delta}\ ;\\
|(D_2\Phi_\be)(\xi)|=
  \left|\frac{\be_2 \Phi_\be(\xi)}{\xi_2} - \Phi_\be(\xi)
  \cdot \frac{(D_2 P)(\xi)}{P(\xi)} \right| \le
  C \left|\frac{\Phi_\be(\xi)}{\xi_2}\right| \le
  \frac{C}{|\xi_2|^{\delta+1}}\ ;\\
 |D_1 D_2 \Phi_\be|=\left|\frac{\be_1\be_2\Phi_{\be'}}{\xi_2}
-D_2\Phi_\be\cdot\frac{D_1 P}{P}-\Phi_\be \(\frac{D_2 D_1 P}{P}-
\frac{D_1 P}{P}\cdot\frac{D_2 P}{P}\)\right|\le
\frac{C}{|\xi_2|^{\delta+1}}.
\end{gather*}

(iii) The case of $|\xi_1|>1$,\ $|\xi_2|\le 1$ is considered in a
similar way.

Thus, $\Phi\in\wt{\cm_1}$ by Theorem~\ref{th1}.
\end{proof}

The following theorem describes wide classes of non-elliptic
systems that are weakly coercive in the isotropic space $\WplRn$,\
$p\in[1,\iy]$. More specifically, the condition $n\ge 2N+1$ of
Theorem~\ref{odnorod} fails for these systems.

\begin{theorem}\label{th_4.3}
Let $\Pj$ be an elliptic system of order $l$ and let
\begin{equation}\label{R_pq}
  R_{uv}(D):=(D_u+i)(D_v+i), \qquad  D_k:=-i\frac{\partial}{\partial x_k}.
\end{equation}
Then the system of operators
\begin{equation}\label{Sjpq}
  S_{juv}(D):= P_j(D)R_{uv}(D),\qquad j\in\{1,\dots,N\},\;
  u,\ v\in\{1,\dots,n\},\; u>v,
\end{equation}
is weakly coercive in
$\overset{\circ}{W}\rule{0pt}{2mm}_p^{l+2}(\R^n)$ for
$p\in[1,\iy]$, but not elliptic.
\end{theorem}

\begin{proof}
The system~\eqref{Sjpq} is not elliptic because the system
$S_{juv}^{l+2}(\xi)=\xi_u \xi_v P^l_j(\xi)$ of its
$(l+2)$-principal parts has a common non-trivial zero at
$\xi^0=(0,\dots,0,1)$.

Further, we choose a monomial $D^\al$ such that $0<|\al|\le l+1$.
Since the variables $\xi_1,\dots,\xi_n$ have 'equal weight'
in~\eqref{R_pq} and~\eqref{Sjpq}, we can assume without loss of
generality that $\al_1>0$. We claim that the following more
stronger estimate holds in place of the weak coercivity
inequality:
\begin{equation}\label{Sj1q}
 \|D^\al f\|_{L^p(\R^n)} \le C_1\sum_{j=1}^N
 \sum_{v=2}^n
 \|S_{j1v}(D)f\|_{L^p(\R^n)}+C_2\|f\|_{L^p(\R^n)},\quad f\in
 C_0^\iy(\R^n),
 \end{equation}
where only operators containing $D_1$ are present in the
right-hand side. To prove~\eqref{Sj1q} it suffices to show that
the functions
\begin{equation}\label{Phi}
\Phi_{\al jv}(\xi):=\chi(\xi) \frac{\xi^\al
\ol{S_{j1v}(\xi)}}{\sum_{q=1}^N\sum_{s=2}^n |S_{q1s}(\xi)|^2}=
\chi(\xi)\frac{\xi^\al (\xi_v-i)\ol{P_j(\xi)}}
{(\xi_1+i)\sum_{q=1}^N |P_q(\xi)|^2 \sum_{s=2}^n (\xi_s^2+1)}
 \end{equation}
are multipliers on $L^p(\R^n)$,\ $p\in[1,\iy]$, whenever $|\al|\le
l+1$,\ $j\in\{1,\dots,N\}$,\ $v\in\{2,\dots,n\}$. Here
$\chi(\xi)\in C_0^\iy(\R^n)$,\ $0\le\chi(\xi)\le 1$, is a 'cutoff'
function equal to one for sufficiently large $|\xi|$ and to zero
in a ball $B_r^n$ containing the compact zero set of the elliptic
system $\Pj$ (see Proposition~\ref{dvei_neodn}, (i)). In fact, if
we prove that $\Phi_{\al jv} \in\cm_1$, then by applying the
inverse Fourier transform to the equalities
$$
  \xi^\al \hat f(\xi)=\sum_{j=1}^N \sum_{v=2}^n
  \Phi_{\al jv}(\xi)S_{j1v}(\xi)\hat f(\xi)+\xi^\al
  (1-\chi(\xi))\hat f(\xi),\qquad f\in C_0^\iy(\R^n),
$$
we obtain the desired estimates~\eqref{Sj1q}.

Assume first that $|\al|<l+1$ and that, as mentioned above,
$\al_1>0$. Then all the following rational fractions belong to
$\cm_1$:
$$
  \frac{\xi_1}{\xi_1+i},\qquad
 \chi(\xi)\frac{\xi_1^{\al_1-1}\xi_2^{\al_2}\dots\xi_n^{\al_n}\ol{P_j(\xi)}}%
 {\sum_{q=1}^N |P_q(\xi)|^2},\qquad
  \frac{\xi_v-i}{\sum_{s=2}^n (\xi_s^2+1)}.
$$
In fact, $\xi_1(\xi_1+i)^{-1}=1-i(\xi_1+i)^{-1}\in\cm_1$ by
formula~\eqref{xi1+i}. It is clear that the two remaining
fractions belong to $\cm_1$ by Theorem~\ref{th1} (or
Proposition~\ref{bez_proizv}), because their denominators are
elliptic polynomials of their variables. Finally, since $\cm_1$ is
an algebra, it follows that $\Phi_{\al jv}\in\cm_1$.

Now, consider the case of $|\al|= l+1$. Clearly,
$$
  \Phi_{\al jv}(\xi)=\chi(\xi)\frac{\xi_1}{\xi_1+i} \cdot
   \frac{\xi^\be (\xi_v-i)\ol{P_j(\xi)}}
  {\sum_{q=1}^N |P_q(\xi)|^2 \sum_{s=2}^n (\xi_s^2+1)}
  =:\frac{\xi_1}{\xi_1+i} \Psi_{\be jv}(\xi),
$$
where $\be:=(\al_1-1,\al_2,\dots,\al_n),\; |\be|=l$. As mentioned
above, $\xi_1(\xi_1+i)^{-1}\in\cm_1$. Therefore, it suffices to
show that $\Psi_{\be jv}\in\cm_1$.

Let $\varkappa$ be the exponent of $\xi_1$ in the product
$\xi^\be\ol{P_j(\xi)}$,\ $\varkappa\le 2l$. We consider two cases.

(i) Suppose $\varkappa<2l$. Clearly, the functions $\Psi_{\be jv}$
are sums of functions of the form
\begin{equation}\label{ell}
\Phi_\ga(\xi):=\chi(\xi)\frac{\xi^\ga}{G(\xi)\sum_{s=2}^n
(\xi_s^2+1)},\qquad |\ga|\le 2l+1,\quad \ga_1\le 2l-1,
\end{equation}
where $G(\xi):=\sum_{q=1}^N |P_q(\xi)|^2$ is an elliptic
polynomial of degree $2l$. We will verify the assumptions of
Proposition~\ref{bez_proizv} for functions~\eqref{ell}. First, we
verify~\eqref{dv1}.

By Proposition~\ref{dvei_neodn}, (ii), if $|\xi|$ is large enough,
then
\begin{equation} \label{ellipt_G}
C_1 |\xi|^{2l}\le |G(\xi)|\le C_2|\xi|^{2l},\qquad C_1,C_2>0.
\end{equation}
Now, the inequality between the geometric mean and mean square
yields
\begin{equation}\label{ld0}
 \prod_{j=1}^n \(1+|\xi_j|\)^\delta \le C
 |\xi|,\qquad\text{where}\quad \delta:=1/n,
\end{equation}
for large $|\xi|$. Since $|\ga|\le 2l+1$ and $\ga_1\le 2l-1$, it
follows that $\ga=\ga'+\td\ga$, where $\ga'$ and $\td\ga$ are
multi-indices such that $\ga_1=\ga_1'\le 2l-1$,\ $|\ga'|\le 2l-1$
and $|\td\ga|\le 2$,\ $\td\ga_1=0$. Then
\begin{equation} \label{ld1}
  |\xi^{\ga'}|\le |\xi|^{|\ga'|}\le |\xi|^{2l-1},\qquad
 |\xi^{\td\ga}|\le \sum_{k=2}^n (\xi_k^2 +1)
\end{equation}
for $|\xi|>1$. Multiplying inequalities~\eqref{ld0}
and~\eqref{ld1} and taking~\eqref{ellipt_G} into account, for
large $|\xi|$ we arrive at relation~\eqref{dv1} for the function
$\Phi_\ga(\xi)$.

Leibniz's formula~\eqref{herm} implies
inequalities~\eqref{ineq_new} for the polynomial
$$
G(\xi)\sum_{s=2}^n (1+\xi_s^2)
$$
because they hold for the elliptic polynomial $G(\xi)$
(see~\cite{BDLM}, \pr\,4) and they obviously hold for
$\sum_{s=2}^n (1+\xi_s^2)$.

Thus, $\Phi_\ga\in\wt{\cm_1}$ by Proposition~\ref{bez_proizv} and
hence $\Psi_{\be jv}\in\wt{\cm_1}$.

(ii) Let $\varkappa=2l$. Then $\xi^\be=\xi_1^l$. Also let
$P_j(\xi)=c_j\xi_1^l+\dots$, where the dots stand for a sum of
monomials containing $\xi_1$ with exponent $<l$. Then
$G(\xi)=\sum_{q=1}^n |c_q|^2 \xi_1^{2l}+\dots$, and we have
$\sum_{q=1}^n |c_q|^2\neq 0$. In view of Theorem~\ref{th1}, the
function
$$
  \Psi'_{\be jv}(\xi):=\chi(\xi)
  \frac{\ol{c_j}}{\sum_{q=1}^N |c_q|^2}\cdot
  \frac{\xi_v-i}{\sum_{s=2}^n (\xi_s^2+1)}
$$
is a multiplier on $L^1(\R^n)$,\; $\Psi'_{\be jv} \in \cm_1$.
Moreover, by step (i),
$$
 \Psi_{\be jv}(\xi)-\Psi'_{\be jv}(\xi)=\chi(\xi)
 \frac{\[\xi_1^l(\ol c_j \xi_1^l+\dots)-
\ol c_j\(\sum_{q=1}^N |c_q|^2\)^{-1} G(\xi)\](\xi_v-i)}
   {G(\xi)\sum_{s=2}^n (\xi_s^2+1)}\in\cm_1,
$$
because the factor in the square brackets contains no monomials
with $\xi_1^{2l}$.
\end{proof}

\begin{cor}\label{th_4.1}
Let $P(D)$ be an elliptic operator of order $l$ and let
$R_{uv}(D)$ be operators of the form~\eqref{R_pq}. Then the system
$\{P(D)R_{uv}(D)\}_{u>v}$ is weakly coercive in
$\overset{\circ}{W}\rule{0pt}{2mm}_p^{l+2}(\R^n)$,\ $p\in[1,\iy]$,
but is not elliptic.
\end{cor}

Next we show that for $p=\iy$ the number of operators $R_{uv}(D)$
in Theorem~\ref{th_4.3} cannot be reduced even if $N=1$.

\begin{prop} \label{N=1}
Suppose that $n\ge 3$, let $P(D)$ be an elliptic operator of order
$l$ and $R_{uv}(D)$ be operators of the form~\eqref{R_pq}. Then if
an arbitrary operator is removed from the system
$\{P(D)R_{uv}(D)\}_{u>v}$ the rest is no longer weakly coercive in
$\overset{\circ}{W}\rule{0pt}{2mm}_\iy^{l+2}(\R^n)$.
\end{prop}

\begin{proof}
Without loss of generality we may assume that the operator
$P(D)R_{12}(D)$ is removed from the system. Assume that the system
$$
\{P(D)R_{uv}(D)\}_{u>v,\ u+v>3}
$$
remains weakly coercive in
$\overset{\circ}{W}\rule{0pt}{2mm}_\iy^{l+2}(\R^n)$. Consider the
'restricted' system
$$
\{P(D)R_{uv}(D)\lceil E\}_{u>v,\ u+v>3},
$$
where $E:=\Span\{\xi_1,\xi_2\}$. It has order $l+1$ and by
Corollary~\ref{suzhenie} is weakly coercive in
$\overset{\circ}{W}\rule{0pt}{2mm}_\iy^{l+2}(\R^n)$. In
particular, this system estimates the operator $D_1^{l+1}$. By
Proposition~\ref{mal_neodn} we obtain
$$
\xi_1^{l+1}=\sum_{u>v,\ u+v>3} \la_{uv}\(P(\xi)R_{uv}(\xi)\lceil
E\)^{l+1}(\xi)= \(P^l(\xi)\lceil E\) \[i\sum_{u>v,\ u+v>3}
\la_{u1} \xi_1+ \la_{u2} \xi_2\].
$$
The polynomial $P^l(\xi)\lceil E$  is elliptic and therefore is
not a multiple of $\xi_1$. Hence $\xi_1^{l+1}$ must divide the
polynomial in the square brackets. However, this contradicts the
relation $l\ge 1$.
\end{proof}

\begin{remark} \label{last_remark}
Proposition~\ref{N=1} shows that Corollary~\ref{cor_10} does not
hold for $N>1$ in the general case. For example, the system
$\{(D_1+i)(D_2+i), (D_3+i)(D_4+i)\}$ is weakly coercive in
$\overset{\circ}{W}\rule{0pt}{2mm}_\iy^2(\R^4)$, but the system
$$
 (D_1^2+\dots+D_4^2)(D_1+i)(D_2+i),\qquad
(D_1^2+\dots+D_4^2)(D_3+i)(D_4+i)
$$
is not weakly coercive in
$\overset{\circ}{W}\rule{0pt}{2mm}_\iy^4(\R^4)$.
\end{remark}



\renewcommand{\refname}{Bibliography}

\bigskip

\bigskip

\noindent \textbf{D.~V.~Limanskii}

\noindent Donetsk National University, Ukraine

\noindent \textit{E-mail:} lim8@telenet.dn.ua

\bigskip

\noindent\textbf{M. M. Malamud}

\noindent Institute of Applied Mathematics and Mechanics of NAS of
Ukraine, Donetsk

\noindent \textit{E-mail:} mmm@telenet.dn.ua

\end{document}